\documentclass[11pt]{article}
\usepackage[T1]{fontenc}
\usepackage{lmodern}
\usepackage[a4paper,margin=1in]{geometry}
\usepackage{amsmath,amssymb,amsfonts,mathtools,amsthm}
\usepackage{graphicx}
\usepackage{booktabs}
\usepackage{tabularx}
\usepackage[authoryear,round]{natbib}
\usepackage{url}
\usepackage{microtype}
\usepackage{setspace}
\usepackage{fancyhdr}
\usepackage{caption}
\usepackage{xr-hyper}
\usepackage[colorlinks=true,linkcolor=black,citecolor=blue,urlcolor=blue]{hyperref}

\graphicspath{{figures/}{figures_eps/}}

\allowdisplaybreaks
\theoremstyle{plain}
\newtheorem{theorem}{Theorem}

\newtheorem{proposition}{Proposition}[section]
\newtheorem{lemma}[proposition]{Lemma}
\theoremstyle{definition}
\newtheorem{remark}{Remark}[section]

\newcommand{\Pp}{\mathbb P}
\newcommand{\Ep}{\mathbb E}
\newcommand{\ind}{\mathbb I}

\newcommand{\bF}{\overline F}
\newcommand{\bG}{\overline G}
\newcommand{\bH}{\overline H}
\newcommand{\bFn}{\overline F_n^{(\mathrm{NA})}}

\newcommand{\keywords}[1]{\par\noindent\textbf{Keywords: }#1\par}
\newcommand{\msc}[1]{\par\noindent\textbf{Mathematics Subject Classification: }#1\par}

\title{\textbf{Weighted and Truncated Tail Index Estimation\\under Random Censoring: A Unified Full-Range Framework}}
\author{Abdelhakim Necir\thanks{Corresponding author. Email: \texttt{ah.necir@univ-biskra.dz}}, Nour Elhouda Guesmia and Djamel Meraghni\\[1ex]
	\small Laboratory of Applied Mathematics, Mohamed Khider University, Biskra, Algeria}
\date{}

\begin{document}
	\maketitle
	
	\begin{abstract}
 We introduce a powered and truncated Nelson--Aalen estimator of the extreme-value index under independent random right censoring of Pareto-type data. The construction combines a power transformation of the estimated survival ratio with the exclusion of a vanishing fraction of the largest observations. It is designed to cover the full range $0<p<1$ of limiting tail uncensoring proportions, including the regime $p\leq1/2$, outside the established classical Nelson--Aalen Gaussian theory. We develop a weighted Gaussian approximation and derive consistency and asymptotic normality with explicit second-order bias and variance. A Weissman-type extreme quantile estimator is also investigated. Simulations and applications to insurance losses and Australian AIDS survival data illustrate the procedure under different censoring intensities.
	\end{abstract}
	
	\msc{62G32; 62G20; 60F17.}
	
	\section{Introduction}\label{sec1}
	Random right censoring occurs when the variable of interest is observed only up to an independent censoring time. Let $X_1,\ldots,X_n$ be independent copies of a positive random variable $X$ with continuous distribution function $F$, and let $C_1,\ldots,C_n$ be independent copies of a positive censoring variable $C$ with continuous distribution function $G$. The two samples are assumed independent. For $1\leq j\leq n$, the observation is
	\begin{equation}
		Z_j=\min(X_j,C_j),\qquad
		\delta_j=\ind_{\{X_j\leq C_j\}}.
	\end{equation}
	Thus $\delta_j=1$ identifies an uncensored observation. We write $H$ for the distribution function of $Z$ and use the notation $\overline L=1-L$ for the survival function associated with any distribution function $L$.
	
	We assume that $F$ and $G$ have Pareto-type tails: for every $x>0$,
	\begin{align}
		\frac{\overline F(tx)}{\overline F(t)}&\longrightarrow x^{-1/\gamma_1}, \label{RVF}\\
		\frac{\overline G(tx)}{\overline G(t)}&\longrightarrow x^{-1/\gamma_2}, \label{RVG}
	\end{align}
	as $t\to\infty$, where $\gamma_1,\gamma_2>0$. Independence implies
	\begin{equation}
		\overline H(x)=\Pp(X>x,C>x)=\overline F(x)\overline G(x).
	\end{equation}
	Multiplying the two regularly varying tails shows that $\overline H\in RV_{-1/\gamma}$, where
	\begin{equation}
		\frac1\gamma=\frac1{\gamma_1}+\frac1{\gamma_2},
		\qquad
		\gamma=\frac{\gamma_1\gamma_2}{\gamma_1+\gamma_2}.
	\end{equation}
	Moreover,
	\begin{equation}
		\overline H^{(1)}(z)
		=\Pp(Z>z,\delta=1)
		=\int_z^\infty \overline G(y)\,dF(y).
	\end{equation}
	Karamata's theorem applied to the last integral yields \[\overline H^{(1)}(z)\sim p\,\overline H(z),\] with
	\begin{equation}
		p=\frac{1/\gamma_1}{1/\gamma_1+1/\gamma_2}
		=\frac{\gamma_2}{\gamma_1+\gamma_2}
		=\frac{\gamma}{\gamma_1},\qquad 0<p<1.
		\label{def-p}
	\end{equation}
	Thus $p$ is the limiting conditional probability that an observation is uncensored, given that it lies in the upper tail; the corresponding censored proportion is $q=1-p=\gamma/\gamma_2$. The regimes $p>1/2$, $p=1/2$, and $p<1/2$ will be referred to as weak censoring, the critical boundary, and strong censoring, respectively.

	The classical Nelson--Aalen Gaussian theory requires $p>1/2$. Our construction combines a power transformation, with exponent $\beta/p$ for fixed $\beta>1$, and a truncation that excludes the finite-count boundary. The power supplies additional tail decay, whereas truncation localizes the empirical approximation to a region with a diverging number of observations. Here, truncation refers solely to the tail estimation procedure, not to the observation mechanism.

	\subsection{Scope of the heavy-tailed censoring assumption}\label{sec:scope}
	The assumption \eqref{RVG} excludes bounded censoring distributions, such as those arising from fixed administrative study horizons. In such settings, stability of the empirical tail uncensoring proportion may be informative over a finite range of thresholds, but it does not establish regular variation or justify the asymptotic theory developed here. Applications outside the stated model must therefore be interpreted as exploratory.

	For a distribution function $L$, let
	\begin{equation}
		L^{\leftarrow}(s)=\inf\{x:L(x)\geq s\},\qquad
		U_L(t)=L^{\leftarrow}(1-1/t),
	\end{equation}
	for $0<s<1$ and $t>1$. The second-order limit theory of the new estimator is formulated in the usual regular-variation framework. We assume that there exist an eventually nonzero function $A_1(t)\to0$ and a parameter $\tau_1\leq0$ such that, for every $x>0$,
	\begin{equation}
		\frac{\overline F(tx)/\overline F(t)-x^{-1/\gamma_1}}{A_1(t)}
		\longrightarrow
		x^{-1/\gamma_1}
		\frac{x^{\tau_1/\gamma_1}-1}{\tau_1\gamma_1},
		\qquad t\to\infty.
		\label{second-orderF}
	\end{equation}
	When $\tau_1=0$, the quotient on the right-hand side is interpreted as $\log(x)/\gamma_1^2$. This is the survival-scale second-order condition used in the asymptotic analysis below.
	
	Let $Z_{1:n}\leq\cdots\leq Z_{n:n}$ denote the order statistics of the observed sample and let $\delta_{[i:n]}$ be the concomitant attached to $Z_{i:n}$. Throughout, $k=k_n$ is an intermediate sequence:
	\begin{equation}
		k\longrightarrow\infty,\qquad \frac{k}{n}\longrightarrow0.
		\label{intermediate-k}
	\end{equation}
	The adapted Hill estimator of \cite{EnFG08}, based on the classical Hill statistic \cite{Hill1975}, is
	\begin{equation}
		\widehat\gamma_{1,k}^{(\mathrm{EFG})}
		=\frac{\widehat\gamma_k^{(H)}}{\widehat p_k},
		\qquad
		\widehat\gamma_k^{(H)}
		=\frac1k\sum_{i=1}^k
		\log\frac{Z_{n-i+1:n}}{Z_{n-k:n}},
		\qquad
		\widehat p_k=\frac1k\sum_{i=1}^k\delta_{[n-i+1:n]}.
		\label{EFG-estimator}
	\end{equation}
	The feasible ratios are defined on the event $\{\widehat p_k>0\}$ and may be assigned an arbitrary value on its complement. Lemma~\ref{lem:p} proves that $\widehat p_k\xrightarrow{\Pp}p>0$; hence \[\Pp(\widehat p_k=0)\leq\Pp(|\widehat p_k-p|\geq p)\to0.\] The asymptotic theory of EFG is available over the full range $0<p<1$. Gaussian approximations for the adapted Hill construction were developed by \cite{BMN2015}. Subsequent work has addressed bias reduction and broader classes of censored tail-index estimators; see, among others, \cite{BBWG2016,BMV2018,BWW2019}. By contrast, the classical product-limit integrals of \cite{WW2014} and \cite{MNS2025} have an established Gaussian theory only under weak censoring, $p>1/2$. The divergence of the MNS limiting variance at $p=1/2$ is the visible statistical symptom of a deeper process-level obstruction.

 Our contribution is a powered, boundary-subtracted Nelson--Aalen tail process and its associated estimator, with Gaussian theory for fixed $0<p<1$. Section~\ref{sec2} states the limit results; Section~\ref{sec3} treats extreme quantiles; Sections~\ref{sec4}--\ref{sec6} present numerical illustrations and conclusions. Proofs are provided in Appendix~\ref{app:proofs}.

	\subsection{Kaplan--Meier and Nelson--Aalen foundations}
	Define the uncensored and censored subdistribution functions by
	\begin{equation}
		H^{(1)}(z)=\Pp(Z\leq z,\delta=1),\qquad
		H^{(0)}(z)=\Pp(Z\leq z,\delta=0),
	\end{equation}
	so that $H=H^{(0)}+H^{(1)}$. Their empirical counterparts are
	\begin{equation}
		H_n^{(j)}(z)=\frac1n\sum_{\ell=1}^n
		\ind_{\{\delta_\ell=j,\,Z_\ell\leq z\}},\qquad j\in\{0,1\}.
	\end{equation}
	Set $H_n=H_n^{(0)}+H_n^{(1)}$ and $\overline H_n=1-H_n$. We also write \[\overline H^{(1)}(z)=\Pp(Z>z,\delta=1).\]
	
	The Kaplan--Meier estimator of $\overline F$ \cite{KM1958} is
	\begin{equation}
		\overline F_n^{(\mathrm{KM})}(z)
		=\prod_{i:Z_{i:n}\leq z}
		\left(1-\frac{\delta_{[i:n]}}{n-i+1}\right),
		\qquad z<Z_{n:n}.
		\label{def-KM}
	\end{equation}
	This is the product-limit object underlying the estimator of \cite{WW2014} and the extreme Kaplan--Meier formulation of \cite{BR2025}.
	
	The Nelson--Aalen estimator of the cumulative hazard of $F$ \cite{Nelson1972} is
	\begin{equation}
		\widehat\Lambda_n(z)
		=\int_0^z\frac{dH_n^{(1)}(v)}{\overline H_n(v-)}
		=\sum_{i:Z_{i:n}\leq z}\frac{\delta_{[i:n]}}{n-i+1},
		\qquad z\leq Z_{n:n},
		\label{def-NA-hazard}
	\end{equation}
	where $\overline H_n(v-)=1-H_n(v-)$. Its exponential survival estimator is
	\begin{equation}
		\overline F_n^{(\mathrm{NA})}(z)
		=\exp\{-\widehat\Lambda_n(z)\}
		=\prod_{i:Z_{i:n}\leq z}
		\exp\left\{-\frac{\delta_{[i:n]}}{n-i+1}\right\}.
		\label{def-NA-survival}
	\end{equation}
	All Nelson--Aalen integrals below are hazard integrals with respect to $d\widehat\Lambda_n$; left limits are displayed explicitly whenever they are needed at jump points.

	\noindent\textbf{Remark on the choice between Kaplan--Meier and Nelson--Aalen.} The Kaplan--Meier and Nelson--Aalen estimators are asymptotically equivalent under weak censoring ($p>1/2$), but they differ in their finite-sample behavior and in their suitability for localization. The Nelson--Aalen estimator has an additive hazard representation that is particularly convenient for martingale control and for the dyadic block decomposition used in the proofs (see Proposition~\ref{prop:mns}). The Kaplan--Meier estimator, being a product, requires logarithmic linearization that introduces additional remainder terms. For these reasons, we follow the Nelson--Aalen route throughout the paper.
	
	\subsection{The classical MNS Nelson--Aalen tail-index estimator}
	Let $Z_k=Z_{n-k:n}$. The MNS estimator of \cite{MNS2025} can be written exactly as the Nelson--Aalen hazard integral
	\begin{equation}
		\widehat\gamma_{1,k}^{(\mathrm{MNS})}
		=\int_{Z_k}^{Z_{n:n}}
		\log\left(\frac{y}{Z_k}\right)
		\frac{\overline F_n^{(\mathrm{NA})}(y-)}
		{\overline F_n^{(\mathrm{NA})}(Z_k)}
		\,d\widehat\Lambda_n(y).
		\label{MNS-integral}
	\end{equation}
	Since the hazard jump at $Z_{n-i+1:n}$ is $\delta_{[n-i+1:n]}/i$, this equals
	\begin{equation}
		\widehat\gamma_{1,k}^{(\mathrm{MNS})}
		=\sum_{i=1}^{k}
		\frac{\delta_{[n-i+1:n]}}{i}
		\prod_{j=i+1}^{k}
		\exp\left\{-\frac{\delta_{[n-j+1:n]}}{j}\right\}
		\log\left(\frac{Z_{n-i+1:n}}{Z_k}\right).
		\label{MNS-discrete}
	\end{equation}
	Thus MNS is a genuine Nelson--Aalen product-limit functional and is distinct from the ratio estimator \eqref{EFG-estimator}.
	
	Under \eqref{RVF}--\eqref{second-orderF}, \eqref{intermediate-k}, and the weak-censoring condition $p>1/2$, \citet{MNS2025} established the following limit. If $\sqrt{k}A_1(U_H(n/k))\to\lambda$, then
	\begin{equation}
		\sqrt{k}\bigl(\widehat\gamma_{1,k}^{(\mathrm{MNS})}-\gamma_1\bigr)
		\xrightarrow{\mathcal D}
		\mathcal N\left(
		\frac{\lambda}{1-\tau_1},
		\frac{p}{2p-1}\gamma_1^2
		\right).
		\label{MNS-limit}
	\end{equation}
	The variance diverges as $p\downarrow1/2$ and does not define a variance for $p<1/2$. Consequently, favorable finite-sample behavior of MNS at $p\leq1/2$ cannot be interpreted as an extension of the established classical Gaussian theory.
	
	For comparison, \cite{WW2014} proposed the Kaplan--Meier integral
	\begin{equation}
		\widehat\gamma_{1,k}^{(\mathrm W)}
		=\sum_{i=1}^{k}\frac{\delta_{[n-i+1:n]}}{i}
		\frac{\overline F_n^{(\mathrm{KM})}(Z_{n-i+1:n}-)}
		{\overline F_n^{(\mathrm{KM})}(Z_k)}
		\log\frac{Z_{n-i+1:n}}{Z_k}.
	\end{equation}
	The explicit product representation studied by \cite{BR2025} is the same Kaplan--Meier functional after reindexing. The Kaplan--Meier and Nelson--Aalen estimators are not identical at finite $n$, but under the classical weak-censoring assumptions they have the same first-order asymptotic bias and variance. The present work follows the Nelson--Aalen route because its additive hazard representation is particularly convenient for localization and martingale control.
	
	\subsection{A new full-range tail-index estimator}
	Let $\Lambda_F=-\log\overline F$ be the cumulative hazard of $F$ and put $a_\beta=\beta/p$. To motivate the estimator, first consider the exact Pareto model $\overline F(tx)/\overline F(t)=x^{-1/\gamma_1}$. Since then $d\Lambda_F(tx)=dx/(\gamma_1x)$, for $R=u_t/t>1$,
	\begin{align}
		a_\beta^2\int_t^{u_t}
		\left\{\frac{\overline F(y)}{\overline F(t)}\right\}^{a_\beta}
		\log\left(\frac yt\right)d\Lambda_F(y)
		&=\frac{a_\beta^2}{\gamma_1}
		\int_1^R x^{-a_\beta/\gamma_1-1}\log x\,dx\nonumber\\
		&=\gamma_1\left[1-R^{-\beta/\gamma}
		\left\{1+\frac{\beta}{\gamma}\log R\right\}\right],
		\label{pareto-truncated-identity}
	\end{align}
	because $a_\beta/\gamma_1=\beta/\gamma$. The bracket tends to one as $R\to\infty$. The uniform convergence theorem and Potter bounds extend this calculation from the exact Pareto model to every $\overline F\in RV_{-1/\gamma_1}$. Consequently, if $t\to\infty$ and $u_t/t\to\infty$, then, for every $\beta>1$,
	\begin{equation}
		\left(\frac{\beta}{p}\right)^2
		\int_t^{u_t}
		\left\{\frac{\overline F(y)}{\overline F(t)}\right\}^{\beta/p}
		\log\left(\frac{y}{t}\right)
		\,d\Lambda_F(y)
		\longrightarrow\gamma_1.
		\label{ass2-tr}
	\end{equation}

	The population identity requires only a positive power. We impose $\beta>1$ as a sufficient condition for the weighted approximation and discretization bounds used below; no claim of necessity is made.

	Let $(m_n)$ be an integer sequence satisfying
	\begin{equation}
		m_n\longrightarrow\infty,
		\qquad \frac{m_n}{k}\longrightarrow0,
		\qquad a_n=\frac{m_n}{n}.
		\label{truncation-scale}
	\end{equation}
	With this vanishing truncation sequence, put
	\begin{equation}
		\widehat u_n=Z_{n-m_n+1:n},\qquad
		\widehat T_{n,m}=\frac{\widehat u_n}{Z_k},\qquad
		\alpha_n=\frac{\beta}{\widehat p_k}.
	\end{equation}
	The proposed estimator is the hazard-integral analogue of \eqref{ass2-tr}:
	\begin{equation}
		\widehat\gamma_{1,k}^{(\mathrm{NA,tr})}(\beta)
		=\alpha_n^2
		\int_{Z_k}^{\widehat u_n}
		\left\{
		\frac{\overline F_n^{(\mathrm{NA})}(y-)}
		{\overline F_n^{(\mathrm{NA})}(Z_k)}
		\right\}^{\alpha_n}
		\log\left(\frac{y}{Z_k}\right)
		\,d\widehat\Lambda_n(y).
		\label{def-est-integral}
	\end{equation}
	Equivalently,
	\begin{equation}
		\widehat\gamma_{1,k}^{(\mathrm{NA,tr})}(\beta)
		=\left(\frac{\beta}{\widehat p_k}\right)^2
		\sum_{i=m_n}^{k}
		\frac{\delta_{[n-i+1:n]}}{i}
		\prod_{j=i+1}^{k}
		\exp\left\{-\frac{\beta}{\widehat p_k}
		\frac{\delta_{[n-j+1:n]}}{j}\right\}
		\log\left(\frac{Z_{n-i+1:n}}{Z_k}\right).
		\label{final_formula_estimator}
	\end{equation}
	The exponent is therefore $-\beta/\widehat p_k$. The index $m_n$ removes the finite-count boundary, whereas the power $\beta/\widehat p_k$ supplies the additional decay required for uniform integrability.

	\noindent\textbf{Practical implementation notes.} In practice, the estimator \eqref{final_formula_estimator} is straightforward to implement: for each candidate threshold $k$, compute $\widehat p_k$ from \eqref{EFG-estimator}, set $\alpha_n=\beta/\widehat p_k$, and evaluate the sum from $i=m_n$ to $i=k$. The truncation index $m_n$ can be chosen automatically as $m_n=\max\{3,\lfloor\log\log k\rfloor\}$ (see Subsection~\ref{choice_mn} below). The right-continuous version \eqref{right-continuous-implementation} is numerically more stable and is asymptotically equivalent by \eqref{right-left-equivalence}; we recommend it for all practical applications.
	
	A minor but important discretization point must be recorded. Put
	\begin{equation}
		h_{i,n}=\frac{\delta_{[n-i+1:n]}}{i},\qquad
		R_{i,k}=\exp\left\{-\sum_{j=i+1}^{k}h_{j,n}\right\},\qquad
		\ell_{i,k}=\log\frac{Z_{n-i+1:n}}{Z_k}.
	\end{equation}
	The exact powered product-limit companion of \eqref{final_formula_estimator} is
	\begin{equation}
		\widetilde\gamma_{1,k}^{(\mathrm{NA,tr})}(\beta)
		=\alpha_n\sum_{i=m_n}^{k}
		R_{i,k}^{\alpha_n}
		\bigl(1-e^{-\alpha_n h_{i,n}}\bigr)\ell_{i,k}.
		\label{powered-companion-discrete}
	\end{equation}
	Define
	\begin{equation}
		Q_n(x)=\frac{\overline F_n^{(\mathrm{NA})}(xZ_k)}
		{\overline F_n^{(\mathrm{NA})}(Z_k)},
		\qquad
		\widehat c_{n,m}=Q_n(\widehat T_{n,m}).
	\end{equation}
	Stieltjes integration by parts then gives the exact identity
	\begin{equation}
		\widetilde\gamma_{1,k}^{(\mathrm{NA,tr})}(\beta)
		=\alpha_n\int_1^{\widehat T_{n,m}}
		\frac{Q_n(x)^{\alpha_n}-\widehat c_{n,m}^{\alpha_n}}{x}\,dx.
		\label{exact-powered-feasible}
	\end{equation}
	The original hazard-integral estimator and its powered companion are asymptotically equivalent:
	\begin{equation}
		\sqrt{k}\left\{
		\widehat\gamma_{1,k}^{(\mathrm{NA,tr})}(\beta)
		-\widetilde\gamma_{1,k}^{(\mathrm{NA,tr})}(\beta)
		\right\}=o_{\Pp}(1).
		\label{discretization-equivalence}
	\end{equation}
	Indeed, \[0\leq x-(1-e^{-x})\leq x^2/2\] gives
	\begin{equation}
		\left|\widehat\gamma_{1,k}^{(\mathrm{NA,tr})}(\beta)
		-\widetilde\gamma_{1,k}^{(\mathrm{NA,tr})}(\beta)\right|
		\leq \frac{\alpha_n^3}{2}
		\sum_{i=m_n}^k h_{i,n}^2R_{i,k}^{\alpha_n}\ell_{i,k}.
		\label{main-jump-bound}
	\end{equation}
	On events whose probabilities tend to one, $\alpha_n$ stays in a compact subset of $(1,\infty)$ and, for some $b>1$, \[ R_{i,k}^{\alpha_n}\leq C(i/k)^b\qquad\text{and}\qquad \ell_{i,k}\leq C\{1+\log(k/i)\}. \] Hence the sum in \eqref{main-jump-bound} is bounded by
	\begin{equation}
		Ck^{-b}\sum_{i=1}^k i^{b-2}\{1+\log(k/i)\}=O(k^{-1}),
	\end{equation}
	which proves \eqref{discretization-equivalence}. The complete argument is recorded in Proposition~\ref{prop:discrete}.
	
	For numerical work it is also common to evaluate the right-continuous Nelson--Aalen ratio at the current jump. In the present notation this gives
	\begin{equation}
		\widehat\gamma_{1,k}^{(\mathrm{NA,tr},R)}(\beta)
		=\alpha_n^2\sum_{i=m_n}^{k}h_{i,n}
		\exp\left\{-\alpha_n\sum_{j=i}^{k}h_{j,n}\right\}
		\ell_{i,k}.
		\label{right-continuous-implementation}
	\end{equation}
	The same quadratic jump bound yields
	\begin{equation}
		\sqrt{k}\left\{
		\widehat\gamma_{1,k}^{(\mathrm{NA,tr},R)}(\beta)
		-\widehat\gamma_{1,k}^{(\mathrm{NA,tr})}(\beta)
		\right\}=o_{\Pp}(1),
		\label{right-left-equivalence}
	\end{equation}
	for every fixed $\beta>1$ and $0<p<1$. Thus the predictable left-limit version used in the proofs and the right-continuous version used in the numerical implementation have the same first-order limit theory.
	
	\subsubsection{Choice of the vanishing truncation index \texorpdfstring{$m_n$}{m-n}}\label{choice_mn}
	The truncation does not remove a fixed fraction of the upper order statistics. Its role is to keep empirical increments above the finite-count scale. Conditions \eqref{truncation-scale} ensure that the truncation endpoint corresponds to a diverging number of observations while the discarded fraction remains asymptotically negligible. For asymptotic normality, we also require that, for some $\zeta\in(0,\beta-1/2)$,
	\begin{equation}
		\sqrt{k}\left(\frac{m_n}{k}\right)^{\beta-\zeta}
		\log\left(\frac{k}{m_n}\right)\longrightarrow0.
		\label{mn-boundary-rate}
	\end{equation}
	The margin $\zeta$ absorbs the slowly varying factors in the endpoint Potter bounds.

	\noindent\textbf{Practical recommendations for $m_n$.} The automatic choice $m_n=\max\{3,\lfloor\log\log k\rfloor\}$ in \eqref{automatic-mn} is conservative and works well in all our simulations. For applications where the sample size is very large ($n\geq 10^4$), one may consider the polynomial choice $m_n=\lfloor k^\rho\rfloor$ with $\rho\in(0,1-1/(2\beta))$, which removes more extreme observations and can improve finite-sample stability at the cost of a slightly larger bias. The boundary-rate condition \eqref{mn-boundary-rate-final} in Theorem~\ref{theorem3} must be verified for the chosen sequence.
	
	One admissible automatic choice is
	\begin{equation}
		m_n=\max\{3,\lfloor\log\log k\rfloor\}.
		\label{automatic-mn}
	\end{equation}
	For \eqref{automatic-mn}, $m_n\to\infty$ and $m_n/k\to0$. Moreover, for any fixed $\zeta<\beta-1/2$,
	\begin{equation}
		\sqrt{k}\left(\frac{m_n}{k}\right)^{\beta-\zeta}
		\log\left(\frac{k}{m_n}\right)
		=k^{1/2-(\beta-\zeta)}m_n^{\beta-\zeta}
		\log\left(\frac{k}{m_n}\right)\to0,
	\end{equation}
	because $\beta-\zeta>1/2$ and $m_n$ is slowly varying. If $m_n=\lfloor k^\rho\rfloor$, the power of $k$ in the same expression is $1/2-(1-\rho)(\beta-\zeta)$. A suitable $\zeta\in(0,\beta-1/2)$ exists whenever
	\begin{equation}
		0<\rho<1-\frac{1}{2\beta}.
		\label{polynomial-mn}
	\end{equation}
	Finally, $m_n\to\infty$ and $m_n/n\to0$ make $m_n$ an intermediate order. The quantile representation for intermediate order statistics therefore gives
	\begin{equation}
		\frac{Z_{n-m_n+1:n}}{U_H(n/m_n)}
		\xrightarrow{\Pp}1.
		\label{endpoint-order-statistic}
	\end{equation}
	
	\subsection{Why the new estimator escapes the weak-censoring barrier}
	To establish the full-range limit law of the estimator, we return to the Nelson--Aalen tail process of \citet{MNS2025}. Put $Z_k=Z_{n-k:n}$ and
	\begin{equation}
		Q_n(x)=
		\frac{\overline F_n^{(\mathrm{NA})}(xZ_k)}
		{\overline F_n^{(\mathrm{NA})}(Z_k)},
		\qquad x\geq1.
		\label{def-Qn}
	\end{equation}
	The unweighted process is
	\begin{equation}
		D_n(x)=\sqrt{k}\{Q_n(x)-x^{-1/\gamma_1}\}.
		\label{def-MNS-process}
	\end{equation}
	The proof in \citet{MNS2025} gives a weighted Gaussian approximation of $D_n$ under $p>1/2$. Before the final decay requirement is imposed, its remainder has the characteristic order
	\begin{equation}
		o_{\Pp}\!\left(x^{(2\eta-p)/\gamma+\varepsilon}\right),
		\qquad \frac14<\eta<\frac12.
	\end{equation}
	Uniform decay therefore requires $2\eta<p$, which is possible only when $p>1/2$. The full-range construction must modify the process before this last step, rather than multiply the already established MNS theorem by an external weight.
	
	Recall the truncation sequence from \eqref{truncation-scale} and define
	\begin{equation}
		u_n=U_H(n/m_n),\qquad
		T_n=\frac{u_n}{Z_k},\qquad
		c_n=Q_n(T_n).
		\label{def-truncation-endpoint}
	\end{equation}
	The divergence of $T_n$ is explicit. Since $Z_k/U_H(n/k)\to1$ in probability and $U_H\in RV_\gamma$,
	\begin{equation}
		T_n=\frac{U_H(n/m_n)}{U_H(n/k)}
		\frac{U_H(n/k)}{Z_k}
		=\left(\frac{k}{m_n}\right)^{\gamma+o(1)}
		\{1+o_{\Pp}(1)\}\xrightarrow{\Pp}\infty.
		\label{Tn-divergence-main}
	\end{equation}
	Furthermore,
	\begin{equation}
		\frac{\overline H(u_n)}{\overline H(Z_k)}
		=\frac{m_n/n}{k/n}\{1+o_{\Pp}(1)\}
		=\frac{m_n}{k}\{1+o_{\Pp}(1)\}.
		\label{smallest-tail-ratio-main}
	\end{equation}
	Thus the smallest tail-probability interval used on $[1,T_n]$ has expected count $k(m_n/k)\{1+o_{\Pp}(1)\}=m_n\{1+o_{\Pp}(1)\}\to\infty$; no finite-count increment enters the localized proof.
	
	For a fixed guiding parameter $\beta>1$, put
	\begin{equation}
		a_\beta=\frac{\beta}{p}.
		\label{def-abeta}
	\end{equation}
	The auxiliary process that drives the estimator theory is
	\begin{equation}
		\mathbb T_n^{(\beta)}(x)
		=\sqrt{k}\left[
		\{Q_n(x)^{a_\beta}-c_n^{a_\beta}\}
		-\{x^{-\beta/\gamma}-T_n^{-\beta/\gamma}\}
		\right],
		\qquad 1\leq x\leq T_n.
		\label{def-powered-process}
	\end{equation}
	The subtraction is performed after taking the power. In particular, $\mathbb T_n^{(\beta)}(T_n)=0$ exactly.
	
 The power is linearized using the population survival ratio and the localized cumulative hazard error, defined by
\[ Q_F(x)=\frac{\bF(xZ_k)}{\bF(Z_k)},\qquad \Delta_n(x)=-\log\{Q_n(x)/Q_F(x)\}. \]
The exact relation \[Q_n(x)^{a_\beta}=Q_F(x)^{a_\beta}e^{-a_\beta\Delta_n(x)}\] separates stochastic and second-order errors. The localized hazard approximation has a remainder of order \[ o_{\Pp}(x^{2\eta/\gamma+\varepsilon}),\qquad 1/4<\eta<1/2. \] Multiplication by $Q_F(x)^{a_\beta}$ gives the weighted order \[ o_{\Pp}(x^{(2\eta-\beta)/\gamma+\varepsilon}) \] (with a Potter margin). The uniform hazard error is $O_{\Pp}(m_n^{-1/2})$, so the exponential Taylor remainder is controlled on the expanding domain. This argument does not apply the classical weak-censoring theorem outside its domain.

	\subsection{Integral representation and tail decomposition}
	Let $T>1$ be deterministic. Equation~\eqref{Tn-divergence-main} implies $\Pp(T_n\leq T)\to0$. Hence, on an event whose probability tends to one, we may write
	\begin{equation}
		\int_1^{T_n}(\cdots)\,dx
		=\int_1^T(\cdots)\,dx+\int_T^{T_n}(\cdots)\,dx.
		\label{fixed-cutoff-decomposition}
	\end{equation}
	The first integral is evaluated from the Wiener representation, while the weighted bound in Theorem~\ref{theorem1} and moment estimates for the Gaussian process control the second as $T\to\infty$. This decomposition is used in the proof of Theorem~\ref{theorem3} (see Appendix~\ref{proof:an}).
	
	For theoretical purposes, the oracle version is obtained by replacing $\alpha_n$ by $a_\beta=\beta/p$. The exact powered companion of the oracle estimator can be written as
	\begin{equation}
		\widetilde\gamma_{1,k}^{(\mathrm{NA,tr},o)}(\beta)
		=a_\beta\int_1^{\widehat T_{n,m}}
		\frac{Q_n(x)^{a_\beta}-Q_n(\widehat T_{n,m})^{a_\beta}}{x}\,dx.
		\label{exact-powered-oracle}
	\end{equation}
	The corresponding identity with the theoretical endpoint $T_n$ provides the exact bridge between the auxiliary process in \eqref{def-powered-process} and the estimator:
	\begin{align}
		\sqrt{k}\left\{
		a_\beta\int_1^{T_n}
		\frac{Q_n(x)^{a_\beta}-c_n^{a_\beta}}{x}\,dx-\gamma_1
		\right\}
		={}&a_\beta\int_1^{T_n}
		\frac{\mathbb T_n^{(\beta)}(x)}{x}\,dx \nonumber\\
		&+\sqrt{k}\left[
		a_\beta\int_1^{T_n}
		\frac{x^{-\beta/\gamma}-T_n^{-\beta/\gamma}}{x}\,dx
		-\gamma_1\right].
		\label{process-estimator-link}
	\end{align}
	The second term in \eqref{process-estimator-link} is explicit. Since $a_\beta/(\beta/\gamma)=\gamma_1$,
	\begin{align}
		&a_\beta\int_1^{T_n}
		\frac{x^{-\beta/\gamma}-T_n^{-\beta/\gamma}}x\,dx-
		\gamma_1\nonumber\\
		&\qquad=-\gamma_1T_n^{-\beta/\gamma}
		\left\{1+\frac{\beta}{\gamma}\log T_n\right\}.
		\label{explicit-pareto-boundary}
	\end{align}
	By \eqref{smallest-tail-ratio-main}, the absolute value of the right-hand side is $O_{\Pp}((m_n/k)^{\beta-\zeta}\log(k/m_n))$ for every small $\zeta>0$. After multiplication by $\sqrt{k}$ it is therefore $o_{\Pp}(1)$ under \eqref{mn-boundary-rate}. The empirical endpoint, the jump discretization, and the replacement of $p$ by $\widehat p_k$ are handled separately (Proposition~\ref{prop:discrete} and Lemmas~\ref{lem:linear}, \ref{lem:second}, and \ref{lem:derivative}).
	
	\section{Full-range asymptotic theory for the new estimator}\label{sec2}
	Put $q=1-p$ and $h=U_H(n/k)$. On a suitable probability space, let $\mathbf W_{n,1}$ and $\mathbf W_{n,2}$ be the rescaled independent Wiener increments associated with the uncensored and censored tail empirical processes. They are centered Gaussian processes with covariance functions
	\begin{equation}
		\Ep\{\mathbf W_{n,1}(s)\mathbf W_{n,1}(t)\}=p\min(s,t),\qquad
		\Ep\{\mathbf W_{n,2}(s)\mathbf W_{n,2}(t)\}=q\min(s,t).
	\end{equation}
	Equivalently, these processes have the same laws as $\sqrt{n/k}\,W_{n,j}((k/n)s)$ for the corresponding Wiener increments $W_{n,j}$. The Gaussian process appearing in the MNS approximation is
	\begin{equation}
		J_n(x)=J_{1n}(x)+J_{2n}(x),
		\label{def-Jn}
	\end{equation}
	where
	\begin{align}
		J_{1n}(x)
		={}&x^{1/\gamma_2}\mathbf W_{n,1}(x^{-1/\gamma})
		-x^{-1/\gamma_1}\mathbf W_{n,1}(1),
		\label{def-J1n}\\
		J_{2n}(x)
		={}&\frac{x^{-1/\gamma_1}}{\gamma}
		\int_1^x u^{1/\gamma-1}
		\left[
		p\mathbf W_{n,2}(u^{-1/\gamma})
		-q\mathbf W_{n,1}(u^{-1/\gamma})
		\right]du.
		\label{def-J2n}
	\end{align}
	Define the weighted Gaussian process
	\begin{equation}
		\mathbb G_n^{(\beta)}(x)
		=\frac{\beta}{p}
		x^{-(\beta-p)/\gamma}J_n(x),
		\qquad x\geq1,
		\label{def-Gbeta}
	\end{equation}
	and the second-order drift
	\begin{equation}
		\mathbb B_\beta(x)
		=\frac{\beta}{p}
		x^{-\beta/\gamma}
		\frac{x^{\tau_1/\gamma_1}-1}{\tau_1\gamma_1},
		\qquad x\geq1,
		\label{def-Bbeta}
	\end{equation}
	with the continuous interpretation $\log(x)/\gamma_1^2$ of the quotient when $\tau_1=0$.
	
	\begin{theorem}[Gaussian approximation supporting the full-range estimator]
		\label{theorem1}
		Assume \eqref{RVF}, \eqref{RVG}, \eqref{second-orderF}, \eqref{intermediate-k}, and \eqref{truncation-scale}, and suppose that $\sqrt{k}A_1(h)=O(1)$. Let $\beta>1$ be fixed. Then, for every
		\begin{equation}
			0<\zeta<\beta-\frac12,
			\label{zeta-range}
		\end{equation}
		we have, for every $0<p<1$,
		\begin{equation}
			\sup_{1\leq x\leq T_n}
			x^{\zeta/\gamma}
			\left|
			\mathbb T_n^{(\beta)}(x)
			-\mathbb G_n^{(\beta)}(x)
			-\sqrt{k}A_1(h)\mathbb B_\beta(x)
			\right|
			\xrightarrow{\Pp}0.
			\label{weighted-process-approximation}
		\end{equation}
		In particular, both the Gaussian term evaluated at $T_n$ and the second-order terminal term are negligible on the weighted scale of \eqref{weighted-process-approximation}.
	\end{theorem}
	
 Under the first-order assumptions alone, the stochastic part of the approximation takes the form
 \[
 \sup_{1\le x\le T_n}x^{\zeta/\gamma}
 \left|\sqrt{k}\{Q_n(x)^{a_\beta}-Q_F(x)^{a_\beta}\}
 -\mathbb G_n^{(\beta)}(x)\right|\xrightarrow{\Pp}0,
 \qquad Q_F(x)=\frac{\bF(xZ_k)}{\bF(Z_k)}.
 \]
 Integration of this first-order approximation, together with the consistency of the pilot exponent and the jump discretization bound, yields consistency without second-order assumptions. All full-range statements concern fixed $p\in(0,1)$, not sequences $p=p_n\to0$.

	Theorem~\ref{theorem1} is the main probabilistic tool behind the full-range estimator. Its proof repeats the decomposition used for the MNS Nelson--Aalen process before the condition $p>1/2$ is imposed. The local MNS remainder retains the order $x^{(2\eta-p)/\gamma+\varepsilon}$; multiplication by the derivative of the power changes it to $x^{(2\eta-\beta)/\gamma+\varepsilon}$. The truncation keeps the expected risk count above $m_n\to\infty$. The proof uses a weighted empirical coupling in probability and martingale bounds for the quadratic hazard remainders. A detailed, term-by-term proof is given in Appendix~\ref{proof:ga}, using the auxiliary results of Appendix~\ref{app:aux} (specifically Lemmas~\ref{lem:coupling}, \ref{lem:linear}, and \ref{lem:second}, and Proposition~\ref{prop:mns}).
	
The centered Gaussian process has the equivalent representation in distribution
\[
 \mathbb G_n^{(\beta)}(x)\overset{\mathcal D}{=}-\frac{\beta}{\sqrt p}x^{-\beta/\gamma}W(x^{1/\gamma}-1),\quad x\ge1,
\]
where $W$ is standard Wiener motion. Thus its pointwise variance is \[ \frac{\beta^2}{p}x^{-2\beta/\gamma}(x^{1/\gamma}-1), \] which decays for $\beta>1/2$. This identity concerns the Gaussian process, not an extension of the estimator theorem to that larger parameter range (see Lemma~\ref{lem:gauss}).

	The exact integral identity in \eqref{process-estimator-link} connects the process approximation with the estimator. For fixed $T>1$, the Gaussian term is evaluated on $[1,T]$, while the contribution over $[T,T_n]$ vanishes as $T\to\infty$. The next two theorems give the resulting consistency and asymptotic normality.
	
	The preliminary tail uncensoring proportion in \eqref{EFG-estimator} satisfies
	\begin{equation}
		\widehat p_k=\frac nk\,\overline H_n^{(1)}(Z_{n-k:n})
		\xrightarrow{\Pp}p
		\label{phat-consistency}
	\end{equation}
	under \eqref{RVF}, \eqref{RVG}, and \eqref{intermediate-k}; this is a consequence of first-order regular variation and is not an additional assumption. See Lemma~\ref{lem:p} and \citet{BMN2015}.
	
	\begin{theorem}[Full-range consistency]\label{theorem2}
		Assume \eqref{RVF}, \eqref{RVG}, \eqref{intermediate-k}, and \eqref{truncation-scale}. For every fixed $\beta>1$,
		\begin{equation}
			\widehat\gamma_{1,k}^{(\mathrm{NA,tr})}(\beta)
			\xrightarrow{\Pp}\gamma_1.
			\label{feasible-consistency}
		\end{equation}
		The same conclusion holds for the right-continuous implementation \eqref{right-continuous-implementation}.
	\end{theorem}
	
	For the feasible central limit theorem, we use the second-order rate of $\widehat p_k$. A convenient sufficient framework is to assume, in addition to \eqref{second-orderF}, that there exist $A_2(t)\to0$ and $\tau_2\leq0$ such that
	\begin{equation}
		\frac{\overline G(tx)/\overline G(t)-x^{-1/\gamma_2}}{A_2(t)}
		\longrightarrow
		x^{-1/\gamma_2}
		\frac{x^{\tau_2/\gamma_2}-1}{\tau_2\gamma_2},
		\qquad x>0,
		\label{second-orderG}
	\end{equation}
	with $\sqrt{k}A_2(h)=O(1)$; for $\tau_2=0$ the quotient is interpreted as $\log(x)/\gamma_2^2$. Under \eqref{second-orderF}--\eqref{second-orderG}, a Stieltjes decomposition of the tail uncensoring proportion (Lemma~2 in the revised proof) yields \[\sqrt{k}(\widehat p_k-p)=O_{\Pp}(1).\]
	
	For later use, define
	\begin{equation}
		N_n(\beta)
		=\frac{\beta(p+\beta-1)}{p}
		\int_0^1s^{\beta-2}\mathbf W_{n,1}(s)\,ds
		+\beta\int_0^1s^{\beta-2}\mathbf W_{n,2}(s)\,ds
		-\frac{\beta}{p}\mathbf W_{n,1}(1).
		\label{def-Nbeta}
	\end{equation}
	It is centered Gaussian and
	\begin{equation}
		\operatorname{Var}\{N_n(\beta)\}
		=\frac{\beta^2}{p(2\beta-1)}.
		\label{variance-Nbeta}
	\end{equation}
	
	\begin{theorem}[Full-range asymptotic normality]\label{theorem3}
		Assume the conditions of Theorem~\ref{theorem1}, the second-order condition \eqref{second-orderG} with $\sqrt{k}A_2(h)=O(1)$, and suppose that, for some $\zeta$ satisfying \eqref{zeta-range},
		\begin{equation}
			\sqrt{k}\left(\frac{m_n}{k}\right)^{\beta-\zeta}
			\log\left(\frac{k}{m_n}\right)\longrightarrow0.
			\label{mn-boundary-rate-final}
		\end{equation}
		Then the new full-range estimator satisfies
		\begin{equation}
			\sqrt{k}\left\{
			\widehat\gamma_{1,k}^{(\mathrm{NA,tr})}(\beta)-\gamma_1
			\right\}
			=\gamma_1N_n(\beta)
			+\frac{\beta}{\beta-p\tau_1}\sqrt{k}A_1(h)
			+o_{\Pp}(1).
			\label{feasible-expansion}
		\end{equation}
		If $\sqrt{k}A_1(h)\to\lambda\in\mathbb R$, then
		\begin{equation}
			\sqrt{k}\left\{
			\widehat\gamma_{1,k}^{(\mathrm{NA,tr})}(\beta)-\gamma_1
			\right\}
			\xrightarrow{\mathcal D}
			\mathcal N\left(
			\frac{\lambda\beta}{\beta-p\tau_1},
			\frac{\gamma_1^2}{p}\frac{\beta^2}{2\beta-1}
			\right).
			\label{feasible-clt}
		\end{equation}
		The right-continuous implementation in \eqref{right-continuous-implementation} has the same expansion and limiting distribution.
	\end{theorem}
	
	The first derivative of the limiting powered Pareto integral with respect to its exponent is zero; consequently, the $k^{-1/2}$ fluctuation of $\widehat p_k$ does not add a first-order term to \eqref{feasible-expansion}. The second-order drift integrates to $\beta/(\beta-p\tau_1)$, while \eqref{mn-boundary-rate-final} removes the deterministic and empirical terminal masses. The complete proof is given in Appendix~\ref{proof:an}.

	\noindent\textbf{Interpretation of the bias and variance.} For $\tau_1<0$, the factor $\beta/(\beta-p\tau_1)$ is increasing in $\beta$. The variance factor $\beta^2/(2\beta-1)$ is also increasing for $\beta>1$. Thus, for fixed $p,\tau_1,\lambda$, increasing $\beta$ increases the variance and, when $\lambda\ne0$ and $\tau_1<0$, the magnitude of the asymptotic bias. The benefit of a larger power lies instead in stronger tail decay and potentially improved finite-sample behavior. The choice of $\beta$ is discussed in Section~\ref{sec4}.

	\begin{remark}[Comparison with EFG]
		The asymptotic variance of EFG is $\gamma_1^2/p$, whereas the proposed estimator has variance
		\begin{equation}
			\sigma_\beta^2=\frac{\gamma_1^2}{p}\frac{\beta^2}{2\beta-1}.
		\end{equation}
		Thus \[\sigma_\beta^2\geq\gamma_1^2/p\] for $\beta>1$, with equality only in the boundary limit $\beta\downarrow1$. For $\beta=1.01$, the relative variance factor is approximately $1.000098$. The practical comparison with EFG is therefore primarily a finite-sample bias and threshold-stability comparison, not a claim of smaller first-order variance.
	\end{remark}
	
	\begin{remark}[Comparison with MNS]
		For $p>1/2$, the MNS variance is \[\sigma_{\mathrm{MNS}}^2=p\gamma_1^2/(2p-1).\] At and below $p=1/2$, the classical MNS Gaussian variance is unavailable: it diverges at the boundary and ceases to define a variance below it. Numerical MNS curves are retained in all simulation scenarios as finite-sample benchmarks, but favorable bias or MSE for $p\leq1/2$ is not interpreted as theoretical validation or as a basis for standard Gaussian confidence intervals.
	\end{remark}
	
\section{Extreme quantile estimation}\label{sec3}
Let $q_{\nu_n}=U_F(1/\nu_n)$, where $\nu_n\to0$. The observable threshold is $t_n=Z_k$, with survival probability $\bF(t_n)$ rather than $k/n$. Estimate this probability by $\widehat v_n=\bFn(t_n)$ and define
\begin{equation}
 \widehat q_{\nu_n}^{(\mathrm{NA,tr})}(\beta)=t_n
 \left(\frac{\widehat v_n}{\nu_n}\right)^{\widehat\gamma_{1,k}^{(\mathrm{NA,tr})}(\beta)}.
 \label{quantile-estimator}
\end{equation}
This Weissman-type extrapolation is exact in the Pareto model when the survival and index are known. The normalization uses $L_n=\log\{\bF(h)/\nu_n\}$, where $h=U_H(n/k)$.
\begin{theorem}[Full-range asymptotic normality of extreme quantiles]\label{theorem:quantile}
 Assume the conditions of Theorem~\ref{theorem3}. Suppose that \[ L_n\to\infty,\qquad L_n/\sqrt{k}\to0, \] and
 \begin{equation}
 \frac{\sqrt{k}}{L_n}\left\{\log(q_{\nu_n}/h)-\gamma_1L_n\right\}\to0.
 \label{quantile-bias-condition}
 \end{equation}
 Then
 \begin{align}
 \frac{\sqrt{k}}{L_n}\left(\frac{\widehat q_{\nu_n}^{(\mathrm{NA,tr})}(\beta)}{q_{\nu_n}}-1\right)
 &=\gamma_1N_n(\beta)+\frac{\beta}{\beta-p\tau_1}\sqrt{k}A_1(h)+o_{\Pp}(1).
 \label{quantile-clt}
 \end{align}
 The same expansion holds for the logarithmic relative error. If $\sqrt{k}A_1(h)\to\lambda$, the limiting law is \[\mathcal N(\lambda\beta/(\beta-p\tau_1),\gamma_1^2\beta^2/[p(2\beta-1)]).\]
\end{theorem}
For $\tau_1<0$, condition \eqref{quantile-bias-condition} follows from the second-order condition and $\sqrt{k}A_1(h)=O(1)$; for $\tau_1=0$ it is retained explicitly. The proof uses the Nelson--Aalen martingale bound and the intermediate threshold bound:
\[ \log\{\widehat v_n/\bF(h)\}=O_{\Pp}(k^{-1/2}),\qquad \log(t_n/h)=O_{\Pp}(k^{-1/2}). \]

\begin{remark}[Plug-in standard error]\label{remark:se}
With $\widehat L_n=\log(\widehat v_n/\nu_n)$, a plug-in standard error (in the extrapolation regime $\widehat L_n>0$) is
\begin{equation}
 \widehat{\mathrm{se}}_{\log q}=\frac{\beta\widehat\gamma_{1,k}^{(\mathrm{NA,tr})}(\beta)}{\sqrt{k\widehat p_k(2\beta-1)}}\widehat L_n.
 \label{se-log-quantile}
\end{equation}
The corresponding interval is
\begin{equation}
 \widehat q_{\nu_n}^{(\mathrm{NA,tr})}(\beta)\exp\{\pm z_{1-\alpha/2}\widehat{\mathrm{se}}_{\log q}\},\label{ci-quantile}
\end{equation}
where $z_{1-\alpha/2}$ is the standard normal quantile. Nominal asymptotic coverage, here and for the uncorrected tail-index intervals, requires negligible asymptotic bias (e.g., $\sqrt{k}A_1(h)\to0$) or appropriate bias correction.
\end{remark}

	\section{Simulation study}\label{sec4}
We compare the proposed estimator at $\beta=1.01,1.30,2$ with EFG and MNS, using signed empirical bias and root mean squared error (RMSE). MNS is included as a numerical benchmark even for $p\leq1/2$, without asserting the classical Gaussian limit there.

\noindent Three Pareto-type models are considered. The first two are the standard Burr and Fr\'{e}chet families,
	\begin{align}
		F_{\mathrm B}(x) &= 1-\left(1+x^{1/\eta}\right)^{-\eta/\zeta}, && x>0,\\
		F_{\mathrm F}(x) &= \exp\{-x^{-1/\zeta}\}, && x>0.
	\end{align}
	In both cases the extreme-value index is $\zeta$. The third model is a logarithmically perturbed Pareto-type distribution. It is defined by $X=\exp(W)$, where $W$ has a Gamma distribution with shape $\alpha$ and scale $\zeta$; equivalently,
	\begin{equation}
		F_{\mathrm{LG}}(x) = \frac{1}{\zeta^{\alpha}\Gamma(\alpha)} \int_0^{\log x} u^{\alpha-1}e^{-u/\zeta}\,du, \qquad x>1.
	\end{equation}
	Its survival function satisfies \[ \overline F_{\mathrm{LG}}(x) \sim \frac{1}{\Gamma(\alpha)} \left(\frac{\log x}{\zeta}\right)^{\alpha-1} x^{-1/\zeta},\qquad x\to\infty, \] so that its extreme-value index is the scale parameter $\zeta$. This model therefore remains regularly varying, but its Pareto approximation contains a slow logarithmic correction and converges more slowly than in the two Hall-type experiments.

	\noindent For the Burr experiments, Burr$(\gamma_1,\eta)$ observations are censored by Burr$(\gamma_2,\eta)$ observations, with $\eta=0.25$. For the Fr\'{e}chet experiments, Fr\'{e}chet$(\gamma_1)$ observations are censored by Fr\'{e}chet$(\gamma_2)$ observations. For the log-gamma experiments, the two logarithmic Gamma variables have the common shape $\alpha=2$ and respective scales $\gamma_1$ and $\gamma_2$; equivalently, $\log X\sim\Gamma(2,\text{rate}=1/\gamma_1)$ and $\log C\sim\Gamma(2,\text{rate}=1/\gamma_2)$. This parameterization ensures that the simulated extreme-value indices are exactly $\gamma_1$ and $\gamma_2$.

	\noindent We take $\gamma_1\in\{0.4,0.7\}$. The censoring levels are $p\in\{0.70,0.50,0.30\}$, representing weak, boundary, and strong censoring, respectively. For each pair $(\gamma_1,p)$, the censoring-tail index is set to $\gamma_2 = \frac{p}{1-p}\gamma_1$, which gives $p=\gamma_2/(\gamma_1+\gamma_2)$. The tuning parameter is examined at $\beta=1.01$, $1.30$, and $2$. The first value is close to the limiting variance benchmark at $\beta=1$, whereas the other two values provide a sensitivity analysis for stronger weighting.

For each of the 18 configurations, $B=2000$ independent samples of size $n=1000$ are generated, using common samples for all estimators. We use $k=20,25,\ldots,300$ and the right-continuous implementation \eqref{right-continuous-implementation}, with $m_n=3$ throughout this grid. MNS uses the left-limit convention in \eqref{MNS-discrete}; the right-limit variant is retained in the accompanying sensitivity files. For each estimator and threshold, let $\mathcal V$ denote the set of replications with finite estimates. The reported measures are
\[
 \mathrm{Bias}=\frac{1}{|\mathcal V|}\sum_{b\in\mathcal V}(\widehat\gamma_{1,k}^{(b)}-\gamma_1),\qquad
 \mathrm{RMSE}=\left\{\frac{1}{|\mathcal V|}\sum_{b\in\mathcal V}(\widehat\gamma_{1,k}^{(b)}-\gamma_1)^2\right\}^{1/2}.
\]
EFG and the proposed estimator are undefined when $\widehat p_k=0$; such replications are counted and excluded from these conditional summaries. At least 1994 of 2000 replications are valid in every displayed curve cell. The Python simulation script, seed, counts and full numerical results accompany the manuscript.

\subsection{Bias and RMSE curves}
Figures~\ref{fig:bias04}--\ref{fig:rmse07} in Appendix~\ref{app:simfig} report the signed bias and RMSE on the common threshold grid, with rows corresponding to the three models and columns to censoring regimes. The horizontal zero line in the bias plots identifies the direction of estimation error. Increasing $\beta$ can change finite-sample bias and RMSE; it increases the asymptotic variance for $\beta>1$. No single tuning value is uniformly optimal across the experiments.
\subsection{Performance at a selected threshold}
For each replication and method, we use a grid adaptation of the Reiss--Thomas stability rule with exponent $0.3$. For $k_j\ge50$, the score is
\[
 j^{-1}\sum_{\ell=1}^j k_\ell^{0.3}|\widehat\gamma_{1,k_\ell}-\operatorname{median}_{r\le j}\widehat\gamma_{1,k_r}|,
 \qquad k_\ell=20+5(\ell-1).
\]
Only prefixes consisting entirely of finite estimates are eligible. The smallest minimizing threshold is selected. This is a grid-based stability rule, not an oracle minimizer of RMSE. For real data, Section~\ref{sec5} adapts the grid to distinct observed thresholds.
Tables~\ref{tab:performance_summary} and \ref{tab:performance_summary07} report selected-threshold bias and RMSE. The three displayed tuning values span near-unit to stronger weighting. The expanded sensitivity grid is summarized below. Inference and coverage are not assessed in this simulation.
\begin{table}[htbp]\centering\small
\caption{Selected-threshold bias and RMSE (bias / RMSE), $\gamma_1=0.4$.}\label{tab:performance_summary}
\begin{tabular}{llrrrrr}\toprule Model & $p$ & EFG & MNS & $\beta=1.01$ & $\beta=1.30$ & $\beta=2$ \\ \midrule
Burr & 0.70 & 0.009 / 0.066 & 0.003 / 0.087 & -0.079 / 0.100 & -0.046 / 0.077 & -0.016 / 0.066 \\
 & 0.50 & 0.033 / 0.087 & -0.068 / 0.134 & -0.074 / 0.104 & -0.038 / 0.084 & -0.005 / 0.078 \\
 & 0.30 & 0.107 / 0.162 & -0.189 / 0.227 & -0.050 / 0.109 & -0.009 / 0.101 & 0.033 / 0.112 \\
Frechet & 0.70 & 0.022 / 0.067 & 0.019 / 0.086 & -0.069 / 0.092 & -0.033 / 0.071 & -0.000 / 0.065 \\
 & 0.50 & 0.044 / 0.092 & -0.056 / 0.130 & -0.063 / 0.098 & -0.026 / 0.081 & 0.008 / 0.080 \\
 & 0.30 & 0.096 / 0.156 & -0.198 / 0.234 & -0.062 / 0.111 & -0.026 / 0.098 & 0.011 / 0.099 \\
Log-gamma & 0.70 & 0.089 / 0.115 & 0.089 / 0.129 & -0.012 / 0.072 & 0.029 / 0.077 & 0.065 / 0.097 \\
 & 0.50 & 0.125 / 0.156 & 0.013 / 0.141 & 0.001 / 0.088 & 0.045 / 0.101 & 0.084 / 0.124 \\
 & 0.30 & 0.232 / 0.283 & -0.169 / 0.230 & 0.030 / 0.126 & 0.075 / 0.146 & 0.121 / 0.178 \\
\bottomrule\end{tabular}\end{table}
\begin{table}[htbp]\centering\small
\caption{Selected-threshold bias and RMSE (bias / RMSE), $\gamma_1=0.7$.}\label{tab:performance_summary07}
\begin{tabular}{llrrrrr}\toprule Model & $p$ & EFG & MNS & $\beta=1.01$ & $\beta=1.30$ & $\beta=2$ \\ \midrule
Burr & 0.70 & 0.002 / 0.107 & -0.006 / 0.147 & -0.142 / 0.173 & -0.089 / 0.133 & -0.044 / 0.110 \\
 & 0.50 & 0.012 / 0.124 & -0.132 / 0.233 & -0.154 / 0.192 & -0.100 / 0.152 & -0.053 / 0.128 \\
 & 0.30 & 0.068 / 0.187 & -0.359 / 0.413 & -0.166 / 0.220 & -0.108 / 0.182 & -0.052 / 0.162 \\
Frechet & 0.70 & 0.042 / 0.119 & 0.036 / 0.154 & -0.120 / 0.161 & -0.056 / 0.124 & -0.000 / 0.113 \\
 & 0.50 & 0.085 / 0.169 & -0.095 / 0.229 & -0.099 / 0.167 & -0.034 / 0.142 & 0.026 / 0.147 \\
 & 0.30 & 0.173 / 0.285 & -0.346 / 0.411 & -0.100 / 0.192 & -0.037 / 0.171 & 0.025 / 0.178 \\
Log-gamma & 0.70 & 0.154 / 0.195 & 0.153 / 0.222 & -0.024 / 0.122 & 0.048 / 0.130 & 0.111 / 0.165 \\
 & 0.50 & 0.218 / 0.270 & 0.027 / 0.239 & 0.006 / 0.152 & 0.080 / 0.172 & 0.147 / 0.213 \\
 & 0.30 & 0.405 / 0.492 & -0.295 / 0.398 & 0.051 / 0.224 & 0.130 / 0.257 & 0.207 / 0.309 \\
\bottomrule\end{tabular}\end{table}

\subsection{Local stability around the selected threshold}
For each replication, let $k_*$ be the threshold selected by the grid rule above and let \[ \mathcal J_*=\{k\in\mathcal K_n: |k-k_*|\leq 0.2 k_*\}, \] where $\mathcal K_n=\{20,25,\ldots,300\}$. We measure local sensitivity by
\[
 S=\left\{\frac{1}{|\mathcal J_*|}\sum_{k\in\mathcal J_*}
 (\widehat\gamma_{1,k}-\widehat\gamma_{1,k_*})^2\right\}^{1/2}.
\]
The score is averaged over replications with a valid selected threshold and finite estimates throughout the neighborhood, and then equally over the 18 configurations. At least 1994 replications per configuration and method are eligible. The neighborhood is intersected with the original grid, so its size varies with the selected threshold. Small $S$ indicates local flatness, not necessarily accuracy. This diagnostic is distinct from RMSE and is not an independent validation of the stability-based selection rule.
We also compute local RMSE by averaging squared estimation errors over the same neighborhood within each eligible replication, then taking the square root of their mean across eligible replications. The table reports equal-weight averages over the 18 configurations. This measures accuracy under local threshold perturbations, whereas $S$ measures flatness alone.
\begin{table}[htbp]\centering\small
\caption{Sensitivity to weighting: mean selected-threshold RMSE, local RMS fluctuation and local RMSE across the 18 configurations.}\label{tab:localstability}
\begin{tabular}{lrrr}\toprule Method & RMSE & Mean $S$ & Local RMSE \\ \midrule
EFG & 0.1740 & 0.0368 & 0.1762 \\
MNS-left & 0.2136 & 0.0183 & 0.2139 \\
MNS-right & 0.2375 & 0.0177 & 0.2369 \\
W (1.01) & 0.1391 & 0.0417 & 0.1402 \\
W (1.05) & 0.1360 & 0.0426 & 0.1376 \\
W (1.1) & 0.1329 & 0.0438 & 0.1352 \\
W (1.2) & 0.1291 & 0.0460 & 0.1329 \\
W (1.3) & 0.1277 & 0.0484 & 0.1327 \\
W (1.4) & 0.1276 & 0.0509 & 0.1339 \\
W (1.5) & 0.1281 & 0.0535 & 0.1356 \\
W (1.75) & 0.1313 & 0.0599 & 0.1418 \\
W (2) & 0.1346 & 0.0664 & 0.1480 \\
\bottomrule\end{tabular}\end{table}
The smallest average selected-threshold RMSE on this grid occurs at $\beta=1.40$, and the smallest local RMSE at $\beta=1.30$. The latter is a descriptive compromise, not a universally optimal tuning value. MNS and EFG have flatter paths on average, but larger local RMSE. The right-limit MNS variant is included only as a sensitivity check. The comparisons use method-specific selected thresholds and the same samples; tuning on this design is exploratory and should not be interpreted as independent validation.
\clearpage

\section{Real-data applications}\label{sec5}
We illustrate the estimators using the insurance loss and Australian AIDS survival datasets. The true tail indices are unknown, so differences between estimates cannot establish bias or accuracy. We use the same nominal threshold range and tuning values as in the simulation study, adapting the candidate grid to the distinct observed values.

\subsection{Threshold selection and tied observations}
For a distinct observed threshold $t$, let $k$ be the number of observations strictly exceeding $t$. We retain all such thresholds with $20\leq k\leq300$ and $t>0$. Let $u_j$ be the distinct observed times in increasing order, $r_j$ the number at risk immediately before $u_j$, and $d_j$ the number of uncensored observations at $u_j$. The grouped Nelson--Aalen jump is $h_j=d_j/r_j$, and the estimated survival is $\exp(-\sum_{u_j\leq t}h_j)$. This convention does not depend on an arbitrary ordering within a tie group. For the proposed estimator we use
\[
 a_k^2\sum_{u_j>t,\ r_j\geq m_n}h_j\exp\left\{-a_k\sum_{t<u_l\leq u_j}h_l\right\}\log(u_j/t),\qquad a_k=\beta/\widehat p_k.
\]
The complete boundary tie group is retained; the effective number of excluded observations may thus be smaller than $m_n-1$. MNS uses the corresponding pre-jump exponential, while EFG uses all observations above $t$, with multiplicities. This grouped implementation is a descriptive extension of the continuous-model theory, not a new limit theorem for discrete distributions.

For the increasing candidate ranks $k_1,\ldots,k_J$, we select the smallest minimizer over $k_j\geq50$ of
\begin{equation}
 j^{-1}\sum_{l=1}^{j} k_l^{0.3}\left|\widehat\xi_{k_l}-\operatorname{median}_{1\leq r\leq j}\widehat\xi_{k_r}\right|.
 \label{kopt}
\end{equation}
Only entirely finite prefixes are eligible. Local stability is measured by the RMS deviation from the selected estimate over the available candidate ranks within $20\%$ of the selected rank. This is a grid-based adaptation of the Reiss--Thomas rule, not an oracle choice of $k$.

\subsection{Insurance loss data}\label{sec:insurance}
The insurance loss dataset, available in the \texttt{copula} R package, contains 1500 losses, of which 34 are right-censored by policy limits. Earlier analyses include \cite{FV98,Klugman99,DPV06}. There are 542 distinct loss amounts. Table~\ref{tab:realdata} reports the estimates reproduced from the original observations. At $\beta=1.30$ the selected threshold is 235000, with $k=53$, $\widehat p_k=0.774$ and an estimated index of 0.577. All selected tail uncensoring fractions exceed one half. This is compatible with the weak-censoring regime, but policy-limit censoring need not satisfy the independent heavy-tailed censoring model.

\subsection{Australian AIDS survival data}\label{sec:aids}
We analyze the 2754 male patients in the \texttt{MASS::Aids2} dataset, using the time from diagnosis to death or end of observation in days and the recorded death indicator; see \cite{RS-94,VR-02,EnFG08,Ndao16}. Of these observations, 1046 are censored; there are 1002 distinct durations. The 27 observations with zero duration remain in the risk sets but do not enter the tail logarithms. For $\beta=1.30$, the selected threshold is 1113 days, with $k=127$, $\widehat p_k=0.283$ and an estimated index of 0.737. The MNS estimate is substantially smaller, at a different threshold. Its selected uncensoring fraction is below one half, so we do not invoke its classical Gaussian limit. Administrative censoring and the finite observation window do not establish the assumptions of the continuous heavy-tailed model.

\begin{table}[htbp]\centering\small
\caption{Recomputed real-data estimates. The uncensoring proportion and local RMS deviation are evaluated at each estimator's selected rank.}\label{tab:realdata}
\begin{tabular}{llrrrr}\toprule Dataset & Estimator & $\widehat k$ & $\widehat p_k$ & Estimate & Local RMS\\\midrule
Insurance & EFG & 53 & 0.774 & 0.668 & 0.036 \\
 & MNS & 73 & 0.836 & 0.773 & 0.050 \\
 & W ($\beta=1.01$) & 52 & 0.769 & 0.470 & 0.052 \\
 & W ($\beta=1.30$) & 53 & 0.774 & 0.577 & 0.043 \\
 & W ($\beta=2.00$) & 53 & 0.774 & 0.662 & 0.058 \\
\midrule
Aids & EFG & 175 & 0.314 & 0.818 & 0.060 \\
 & MNS & 50 & 0.260 & 0.163 & 0.019 \\
 & W ($\beta=1.01$) & 300 & 0.417 & 0.579 & 0.020 \\
 & W ($\beta=1.30$) & 127 & 0.283 & 0.737 & 0.037 \\
 & W ($\beta=2.00$) & 127 & 0.283 & 0.788 & 0.050 \\
\bottomrule\end{tabular}
\end{table}

\begin{figure}[htbp]\centering
\includegraphics[width=.49\textwidth]{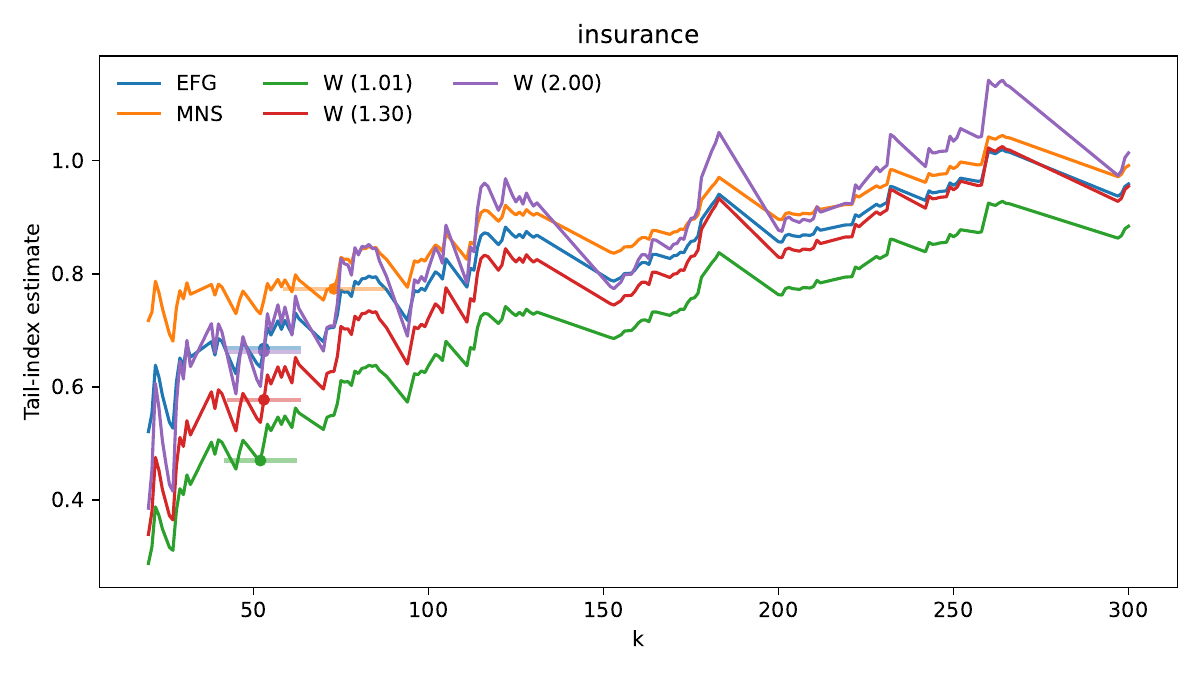}\hfill
\includegraphics[width=.49\textwidth]{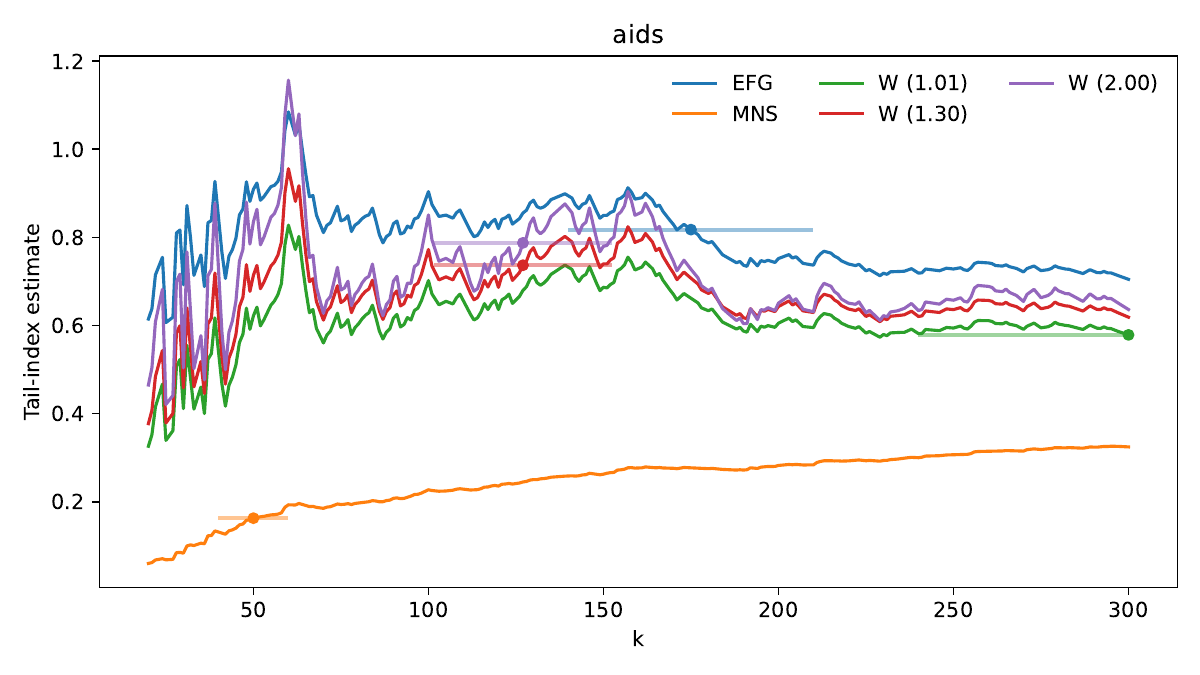}
\caption{Tail-index paths on the distinct-threshold grid for insurance losses (left) and AIDS survival durations (right). Points indicate selected ranks; short horizontal segments show the local threshold neighborhoods.}\label{fig:realpaths}
\end{figure}

\begin{figure}[htbp]\centering
\includegraphics[width=.49\textwidth]{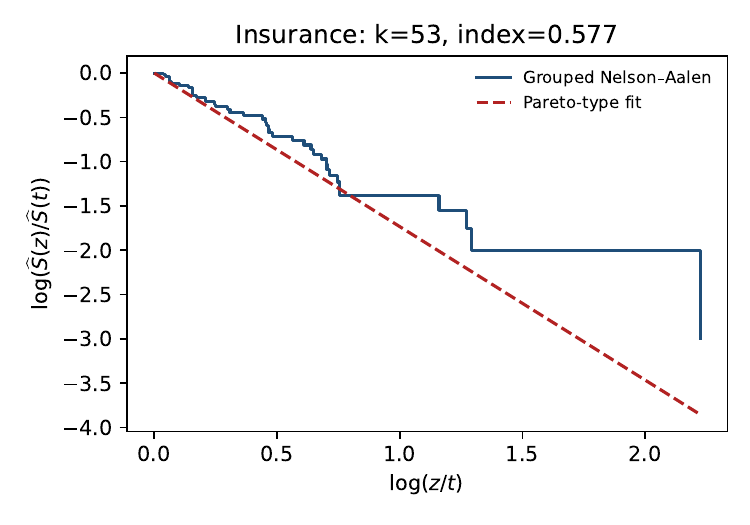}\hfill
\includegraphics[width=.49\textwidth]{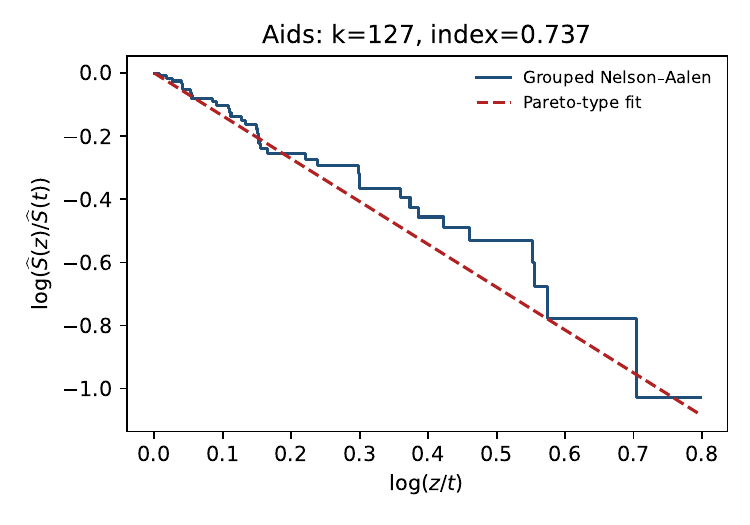}
\caption{Conditional log-survival diagnostics for insurance losses (left) and AIDS survival durations (right). The step function is the grouped Nelson--Aalen log-survival ratio above the selected threshold; the dashed line is the Pareto-type fit with slope $-1/\widehat\gamma_1$ for $\beta=1.30$.}\label{fig:paretofit}
\end{figure}

Figure~\ref{fig:paretofit} compares $\log\{\widehat{\overline F}_{n}^{(\mathrm{NA})}(z)/\widehat{\overline F}_{n}^{(\mathrm{NA})}(t_n)\}$ with $-\log(z/t_n)/\widehat\gamma_1$ for $z\geq t_n$. The two curves are anchored at the same threshold. For insurance, the fit follows the general tail trend, although the empirical log-survival at the largest observed loss is above the fitted line (approximately $-3.00$ versus $-3.85$). For AIDS, the two curves are close over the observed range, with terminal values approximately $-1.03$ and $-1.08$. The latter agreement is over a short log-exceedance range and does not justify extrapolation far beyond the observed durations. These plots are descriptive diagnostics based on the same data used to estimate the slope; they are neither independent goodness-of-fit tests nor evidence for independent heavy-tailed censoring.

\subsection{Sensitivity and extreme quantiles}\label{sec:inferential}
For $\beta=1.30$, reducing the upper grid bound from 300 to 200 leaves both selected estimates unchanged. Increasing it to 500 leaves the insurance estimate unchanged but selects $k=500$ for AIDS, with index 0.608 and $\widehat p_k=0.500$. Increasing the lower eligible selection rank from 50 to 75 instead changes the insurance result to $k=75$ and index 0.707. Starting the candidate grid at 30 rather than 20 changes the AIDS result to $k=126$ and index 0.720. Hence the selected estimates are sensitive to grid design, despite locally modest RMS deviations. Figures~\ref{fig:realpaths} and \ref{fig:paretofit} display the selected paths and the associated Pareto-type diagnostics for the reference grid.

For illustration, we evaluate the Weissman-type formula at $\nu_n=1/n$ for $\beta=1.30$. The extrapolation factors $n\widehat v_n$ are 61.92 and 443.31, respectively. The resulting values are approximately $2.542\times10^6$ in the insurance loss units and $9.922\times10^4$ days for AIDS. The latter far exceeds the largest observed duration (2470 days), highlighting the risk of mechanically extrapolating a Pareto model in this setting. It is not a clinical survival prediction. Even in the insurance example the estimate exceeds the largest observed loss (2173595). We report these as exploratory numerical extrapolations only: neither the model assumptions, the negligible-bias condition, nor the asymptotic extrapolation regime is established empirically. We therefore do not report nominal confidence intervals as validated inferential statements for these datasets.

\section{Conclusion}\label{sec6}
We introduced a powered and truncated Nelson--Aalen estimator of the extreme-value index for fixed $0<p<1$. The weighted Gaussian approximation explains how the power restores tail decay while truncation keeps the risk count away from the finite-count boundary. Consistency requires only first-order regular variation; the second-order framework gives explicit bias and variance and supports extreme-quantile extrapolation.

The proposed variance approaches that of EFG as $\beta\downarrow1$; any finite-sample advantage must be assessed separately. MNS comparisons below $p=1/2$ are descriptive rather than Gaussian inference. Both data illustrations are exploratory: finite policy limits and administrative censoring need not satisfy the heavy-tailed censoring model. Extending the theory to bounded or discrete censoring and developing bias correction and joint tuning of $k,m_n,\beta$ remain open problems.
\label{sec:limitations}

	\clearpage
\appendix
\section{Proofs of the main results}\label{app:proofs}
We prove Theorems 1--4, followed by nine auxiliary results. Assume independent, continuous lifetimes and censoring times, with fixed $p\in(0,1)$ and $\beta>1$. Write $q=1-p$, $\gamma=p\gamma_1$, $a=\beta/p$, $t_n=Z_{n-k:n}$, $h=U_H(n/k)$, $u_n=U_H(n/m_n)$, $T_n=u_n/t_n$ and $\widehat T_n=Z_{n-m_n+1:n}/t_n$. As usual, $k\to\infty$, $k/n\to0$, $m_n\to\infty$ and $m_n/k\to0$. Random intervals are considered on the event of ordered endpoints, with probability tending to one. Define
\[
\begin{gathered}
Q_F(x)=\frac{\bF(xt_n)}{\bF(t_n)},\qquad
Q_n(x)=\frac{\bFn(xt_n)}{\bFn(t_n)},\qquad
s_x=\frac{\bH(xt_n)}{\bH(t_n)},\\
\Delta_n(x)=-\log\{Q_n(x)/Q_F(x)\}.
\end{gathered}
\]
We use $\Lambda_F=-\log\bF$, and $\overline H^{(j)}(z)=\Pp(Z>z,\delta=j)$, $j=0,1$. The two independent Wiener processes $\mathbf W_{n,1},\mathbf W_{n,2}$ have covariance kernels $p(s\wedge t)$ and $q(s\wedge t)$, respectively. All probability approximations below may be constructed on an enlarged probability space. Constants may depend on the fixed parameters. Potter bounds are used with arbitrarily small positive margins (see \citealp{deHF06}).

\subsection{Proof of Theorem 1: Weighted Gaussian approximation}\label{proof:ga}
The exact hazard identity is\[ Q_n=Q_F\exp(-\Delta_n). \] Lemma~\ref{lem:linear} and Proposition~\ref{prop:mns} give, for $0<\zeta<\beta-1/2$,
\begin{align*}
\sqrt{k}(Q_n^a-Q_F^a)&=-aQ_F^a\sqrt{k}\Delta_n+o_{\Pp}(x^{-\zeta/\gamma})\\
&=aQ_F^{a-1}J_n(x)+aQ_F^{a-1}R_n(x)+o_{\Pp}(x^{-\zeta/\gamma}),
\end{align*}
where the little-oh notation is uniform on $[1,T_n]$ in the indicated weighted norm. The proposition gives $R_n=o_{\Pp}(x^{(2\eta-p)/\gamma+\epsilon})$ for $1/4<\eta<1/2$. Choose $2\eta+\zeta<\beta$ and a sufficiently small Potter margin. Since $Q_F^{a-1}\le Cx^{-(\beta-p)/\gamma+\epsilon}$, the weighted remainder tends to zero. On compact $x$-intervals, $Q_F^{a-1}\to x^{-(\beta-p)/\gamma}$ uniformly in probability. Outside a fixed compact interval the product of either coefficient with $J_n$ is bounded by $O_{\Pp}(x^{(1/2-\beta)/\gamma+2\epsilon})$ (Proposition~\ref{prop:mns}). Letting the compact interval expand yields
\begin{equation}\label{eq:stochastic}
\sup_{1\le x\le T_n}x^{\zeta/\gamma}\left|\sqrt{k}(Q_n^a-Q_F^a)-\mathbb G_n^{(\beta)}(x)\right|=o_{\Pp}(1),
\end{equation}
where\[\mathbb G_n^{(\beta)}(x)=ax^{-(\beta-p)/\gamma}J_n(x).\] This requires first-order assumptions only. Lemma~\ref{lem:second} adds the drift $\sqrt{k}A_1(h)\mathbb B_\beta(x)$, where\[\mathbb B_\beta(x)=ax^{-\beta/\gamma}d_{\tau_1}(x),\qquad d_\tau(x)=\frac{x^{\tau/\gamma_1}-1}{\tau\gamma_1},\] interpreted as $\log x/\gamma_1^2$ for $\tau=0$. By the Gaussian bound and the deterministic drift envelope,
\[
 T_n^{\zeta/\gamma}\big(|\mathbb G_n^{(\beta)}(T_n)|+\sqrt{k}|A_1(h)\mathbb B_\beta(T_n)|\big)=o_{\Pp}(1).
\]
Subtracting the expansion at $T_n$ contributes at most its weighted remainder at $T_n$. This proves Theorem~\ref{theorem1} for
\[
\mathbb T_n^{(\beta)}(x)=\sqrt{k}\{Q_n(x)^a-Q_n(T_n)^a-x^{-\beta/\gamma}+T_n^{-\beta/\gamma}\}.
\]

\subsection{Proof of Theorem 2: Full-range consistency}\label{proof:cons}
For $\alpha>0$, define
\[
 \Phi_n(\alpha,T)=\alpha\int_1^T\{Q_n(x)^\alpha-Q_n(T)^\alpha\}\,\frac{dx}{x}.
\]
Integrating \eqref{eq:stochastic}, using $\int_1^\infty x^{-1-\zeta/\gamma}dx=\gamma/\zeta$ and Lemma~\ref{lem:gauss}, shows that the integrated stochastic difference is $O_{\Pp}(k^{-1/2})$. The population integral $a\int_1^{T_n}Q_F(x)^a dx/x$ tends to $\gamma_1$ by compact convergence and Potter domination. Also $Q_n(T_n)^a\log T_n=o_{\Pp}(1)$ by Proposition~\ref{prop:mns} and Potter bounds. Thus\[\Phi_n(a,T_n)\to_{\Pp}\gamma_1.\] Lemma~\ref{lem:p} gives $\alpha_n=\beta/\widehat p_k\to_{\Pp}a$. Lemma~\ref{lem:derivative} and the mean-value theorem give\[\Phi_n(\alpha_n,T_n)-\Phi_n(a,T_n)=o_{\Pp}(1).\] Proposition~\ref{prop:discrete} gives the same conclusion for the empirical endpoint and for both jump conventions. No second-order or root-$k$ boundary condition is used in this proof.

\subsection{Proof of Theorem 3: Full-range asymptotic normality}\label{proof:an}
The deterministic Pareto integral minus $\gamma_1$ equals
\[
 D_\beta(T)=-\gamma_1T^{-\beta/\gamma}\{1+(\beta/\gamma)\log T\}.
\]
Lemma~\ref{lem:tail} and the boundary condition in Theorem 3 imply $\sqrt{k}D_\beta(T_n)=o_{\Pp}(1)$: use a Potter margin $0<\epsilon<\zeta$ in the exponent of $m_n/k$. Therefore Theorem 1 yields
\begin{align*}
\sqrt{k}\{\Phi_n(a,T_n)-\gamma_1\}
 &=a\int_1^{T_n}\mathbb G_n^{(\beta)}(x)\frac{dx}{x}
 \\ &\quad+a\sqrt{k}A_1(h)\int_1^{T_n}\mathbb B_\beta(x)\frac{dx}{x}+o_{\Pp}(1).
\end{align*}
The integrated error is bounded by $a\gamma/\zeta$ times the weighted supremum error. Lemma~\ref{lem:gauss} evaluates the Gaussian integral as $\gamma_1N_n(\beta)+o_{\Pp}(1)$. The deterministic tail beyond $T_n$ vanishes because $\sqrt{k}A_1(h)=O(1)$ and $|\mathbb B_\beta(x)|\le Cx^{-\beta/\gamma}(1+\log x)$. Direct integration gives
\[
 a\int_1^\infty\mathbb B_\beta(x)\frac{dx}{x}
 =\frac{a^2}{\tau_1\gamma_1}\left(\frac{\gamma}{\beta-p\tau_1}-\frac{\gamma}{\beta}\right)
 =\frac{\beta}{\beta-p\tau_1}.
\]
For $\tau_1=0$, the integral is $a^2\gamma^2/(\gamma_1^2\beta^2)=1$. Lemmas~\ref{lem:p} and \ref{lem:derivative} show that replacing $a$ by $\alpha_n$ costs $o_{\Pp}(k^{-1/2})$. Proposition~\ref{prop:discrete} shows the same for the empirical endpoint and both discretizations. We obtain the expansion of Theorem 3. The stated normal limit follows from the exact Gaussian variance computed in Lemma~\ref{lem:gauss} and Slutsky's theorem.

\subsection{Proof of Theorem 4: Extreme quantile estimation}\label{proof:quant}
Write\[\begin{gathered}v=\bF(h),\qquad \widehat v_n=\bFn(t_n),\qquad L_n=\log(v/\nu_n),\\ B_n=\log(q_{\nu_n}/h)-\gamma_1 L_n.\end{gathered}\] We first show that $\log(\widehat v_n/v)=O_{\Pp}(k^{-1/2})$. Let $h^+=U_H(2n/k)$, so $\Pp(t_n\le h^+)\to1$. The relative empirical risk bound established in the proof of Proposition~\ref{prop:mns}, applied from zero to $h^+$ with minimal expected risk count $k/2$, gives $Y_n(z)\ge\bH(z)/2$ simultaneously with probability tending to one. Here $Y_n=\bH_n(z-)$ and $M_n=H_n^{(1)}-\int Y_n d\Lambda_F$ is the counting-process martingale. The predictable localized integrand $\ind_{\{Y_n\ge\bH/2\}}/Y_n$ (zero for $Y_n=0$) has integrated quadratic variation at most
\[
 \frac2n\int_0^{h^+}\frac{d\Lambda_F}{\bH}\le\frac2{n\bH(h^+)}=\frac4k.
\]
Doob's inequality and evaluation at $t_n$ imply $\log(\bFn(t_n)/\bF(t_n))=O_{\Pp}(k^{-1/2})$. Monotonicity of $\Lambda_F$ and $\Lambda_G$ implies
\[
 |\log(\bF(t_n)/v)|\le |\log(\bH(t_n)/\bH(h))|=O_{\Pp}(k^{-1/2}).
\]
By compact uniformity of the second-order condition at $t_n/h\to_{\Pp}1$ and $\sqrt{k}A_1(h)=O(1)$, \[\log(\bF(t_n)/v)=-\gamma_1^{-1}\log(t_n/h)+o_{\Pp}(k^{-1/2}).\] Hence $\log(t_n/h)=O_{\Pp}(k^{-1/2})$.

For $\widehat\gamma=\widehat\gamma_{1,k}^{(\mathrm{NA,tr})}(\beta)$, the corrected extrapolation has the exact identity
\begin{equation}\label{eq:quantile}
 \begin{aligned}\log(\widehat q_{\nu_n}/q_{\nu_n})&=L_n(\widehat\gamma-\gamma_1)+\log(t_n/h)\\ &\quad+\widehat\gamma\log(\widehat v_n/v)-B_n.\end{aligned}
\end{equation}
Multiply by $\sqrt{k}/L_n$. The last three terms tend to zero in probability by consistency, $L_n\to\infty$ and the assumed bias condition. Theorem 3 proves the logarithmic expansion. Since $L_n/\sqrt{k}\to0$, Taylor's formula for the exponential gives the same expansion for the relative error, with scaled remainder $O_{\Pp}(L_n/\sqrt{k})=o_{\Pp}(1)$.

If $\tau_1<0$, the extrapolation-bias assumption follows from the second-order condition: for $1\le y\le2$, the log-survival ratio differs from $-\log y/\gamma_1$ by at most $C|A_1(t)|$ uniformly at all sufficiently large $t$. For any $x\ge1$, telescope along $t,2t,\ldots,2^jt$ with a final factor in $[1,2]$. Potter bounds for $|A_1|\in RV_{\tau_1/\gamma_1}$ show that the sum of the errors is bounded by $C|A_1(t)|\sum_{j\ge0}2^{j\tau_1/(2\gamma_1)}$. Taking $t=h$, $x=q_{\nu_n}/h$ yields $B_n=O(|A_1(h)|)$ and thus $\sqrt{k}B_n/L_n\to0$. For $\tau_1=0$ this argument is unavailable and the bias condition is retained. Finally $\widehat L_n/L_n\to_{\Pp}1$ for $\widehat L_n=\log(\widehat v_n/\nu_n)$; the plug-in standard error is consistent on its natural scale. Uncorrected nominal coverage requires vanishing scaled bias.

\section{Auxiliary results}\label{app:aux}
\setcounter{proposition}{0}\renewcommand{\theproposition}{\arabic{proposition}}
\begin{lemma}[Tail subdistributions and truncation scales]\label{lem:tail}
Under the first-order assumptions, $\overline H^{(1)}/\bH\to p$, $\overline H^{(0)}/\bH\to q$. Furthermore
\[
 b_n=n\bH(t_n)/k=1+O_{\Pp}(k^{-1/2}),\quad t_n/h\to_{\Pp}1,
\quad r_n=\bH(u_n)/\bH(t_n)=\frac{m_n}{k b_n}.
\]
Both $T_n$ and $\widehat T_n$ diverge in probability; $\widehat u_n/u_n\to_{\Pp}1$ with $\widehat u_n=t_n\widehat T_n$. For each $\epsilon>0$, uniformly for $x\ge1$ with probability tending to one,
\[
\begin{gathered}
C^{-1}x^{-p/\gamma-\epsilon}\le Q_F(x)\le Cx^{-p/\gamma+\epsilon} ,\\
C^{-1}x^{-1/\gamma-\epsilon}\le s_x\le Cx^{-1/\gamma+\epsilon}.
\end{gathered}
\]
\end{lemma}
\begin{proof}
Independence gives $\bH=\bF\bG$. The conditional tail probability of a failure is $-\int_1^\infty g_t\,df_t$, where $f_t(x)=\bF(tx)/\bF(t)$ and $g_t(x)=\bG(tx)/\bG(t)$. Both converge uniformly on $[1,\infty)$ to the corresponding Pareto tails by monotonicity and compact convergence. Stieltjes integration by parts gives the limit $\gamma/\gamma_1=p$. The remaining tail proportion is $q$. The probability integral transform and beta moments of uniform order statistics give the asserted rate for $b_n$, as well as
\[
 \sup_{m_n\le i\le k}\left|\frac{n\bH(Z_{n-i+1:n})}{i}-1\right|=o_{\Pp}(1).
\]
For the latter, the uniform order statistics have the distribution of normalized partial sums of independent unit exponentials; the strong law gives uniform relative convergence over all $i\ge m_n$, and the normalizing sum divided by $n$ tends to one. Quantile regular variation gives the threshold ratios, including $\widehat u_n/u_n\to_{\Pp}1$. Potter's inequalities imply the displayed bounds and $\log(T_n\vee\widehat T_n)=O_{\Pp}(\log(k/m_n))$. They also give $(T_n\wedge\widehat T_n)^{-\beta/\gamma}=O_{\Pp}((m_n/k)^{\beta-\epsilon})$ for every small $\epsilon>0$ (adjust the Potter margin).
\end{proof}

\begin{lemma}[Consistency and root-$k$ rate of $\widehat p_k$]\label{lem:p}
Under the first-order assumptions $\widehat p_k\to_{\Pp}p$. Under the second-order conditions on both tails and $\sqrt{k}A_j(h)=O(1)$, $j=1,2$, we have\[\sqrt{k}(\widehat p_k-p)=O_{\Pp}(1).\]
\end{lemma}
\begin{proof}
Let $p(t)=\overline H^{(1)}(t)/\bH(t)$. The exact decomposition is
\[
 \widehat p_k-p=\frac nk(\overline H_n^{(1)}-\overline H^{(1)})(t_n)+p(t_n)(b_n-1)+p(t_n)-p.
\]
The first two terms are $O_{\Pp}(k^{-1/2})$ by Lemmas~\ref{lem:coupling} and \ref{lem:tail}; the last tends to zero by Lemma~\ref{lem:tail}. For the rate, put $f_0=x^{-1/\gamma_1}$, $g_0=x^{-1/\gamma_2}$ and use the notation of the previous proof. The second-order Potter bounds imply
\[
\sup_{x\ge1}|f_t-f_0|=O(|A_1(t)|)
\quad\text{and}\quad
\sup_{x\ge1}|g_t-g_0|=O(|A_2(t)|),
\]
including zero second-order indices. Since
\[
 p(t)-p=-\int_1^\infty(g_t-g_0)df_t-\int_1^\infty g_0\,d(f_t-f_0),
\]
Stieltjes integration by parts in the second term bounds the absolute difference by the sum of the two suprema. Regular variation of $|A_j|$ and $t_n/h\to_{\Pp}1$ complete the proof. This gives a direct version of the pilot-rate argument used in \citet{BMN2015}.
\end{proof}

\begin{lemma}[Weighted empirical coupling]\label{lem:coupling}
Define $\alpha_{n,j}(z)=(n/\sqrt{k})(\overline H_n^{(j)}(z)-\overline H^{(j)}(z))$ and $z_h(s)=U_H(n/(ks))$. For $0<d<1/2$, on an enlarged probability space,
\[
 \sup_{0<s\le2}s^{-d}|\alpha_{n,j}(z_h(s))-\mathbf W_{n,j}^*(s)|=o_{\Pp}(1),\qquad j=0,1,
\]
where $\mathbf W_{n,1}^*=\mathbf W_{n,1}$ and $\mathbf W_{n,0}^*=\mathbf W_{n,2}$. The same bound on $0<s\le1$ holds with $z_h(s)$ replaced by $z_h(b_ns)$.
\end{lemma}
\begin{proof}
The variables\[U_i=\delta_i H^{(1)}(Z_i)+(1-\delta_i)(\theta+H^{(0)}(Z_i)),\qquad \theta=\Pp(\delta=1),\] are uniform on $[0,1]$. The marked tails correspond to nested intervals ending at $\theta$ and 1. Their lengths at $z_h(s)$ are $(k/n)a_{n,j}(s)$, where $a_{n,1}+a_{n,0}=s$, $a_{n,1}\to ps$ and $a_{n,0}\to qs$ uniformly on $[0,2]$ by Lemma~\ref{lem:tail}. Each $a_{n,j}$ is nondecreasing and 1-Lipschitz.

For clarity we establish compact functional convergence by poissonization. On an auxiliary space take an independent Poisson($n$) sample size $N$ and an infinite i.i.d. sequence of uniform variables. Counts in the two disjoint marked-tail intervals form independent rate-one Poisson processes at times $ka_{n,j}(s)$. The functional central limit theorem for Poisson processes, followed by the uniform deterministic time changes $a_{n,j}$, gives convergence to the two stated Wiener processes in the Skorokhod space $D[0,2]^2$ equipped with the $J_1$ topology (choose right-continuous versions). The difference between the fixed-$n$ and Poisson counts, uniformly over $s\le2$, is bounded by the number of the $|N-n|$ added or deleted observations in the two maximal tail intervals. Its expected value is at most $2(k/n)\Ep|N-n|\le 2k/\sqrt n$. After division by $\sqrt{k}$ it is $O(\sqrt{k/n})$. Markov's inequality gives the same convergence for the fixed-size centered counts. The Poisson functional limit itself follows from independent increments, the ordinary central limit theorem for Poisson increments and the usual partial-sum functional central limit theorem applied to unit-time increments; the within-unit maximum is negligible by the Poisson exponential tail bound.

To control zero, symmetrize the centered count over the nested marked intervals $A_j(s)$ and condition on the observations. For $0<u\le2$, the Rademacher sums along $A_j(s)$ are partial sums of at most $N_j(u)$ independent signs. Doob's inequality gives
\[
 \Ep\sup_{0<s\le u}|\alpha_{n,j}(z_h(s))|^2\le\frac{16}{k}\Ep N_j(u)\le16u.
\]
On dyadic intervals $(u2^{-l-1},u2^{-l}]$ use Markov's inequality and sum the geometric series to get
\begin{equation}\label{eq:dyadic}
 \Pp\left(\sup_{0<s\le u}s^{-d}|\alpha_{n,j}(z_h(s))|>\epsilon\right)\le C_d\epsilon^{-2}u^{1-2d}.
\end{equation}
The same bound holds for the Wiener processes. Skorokhod representation of the joint compact convergence gives a coupling with uniform convergence in probability, since the limit paths are continuous. On $[u,2]$ it implies convergence in the weighted norm; \eqref{eq:dyadic} gives convergence on $(0,u]$ by first taking $n\to\infty$, then $u\downarrow0$. The coupling may be transferred to the original observations by the conditional law of the coupled Wiener paths given the empirical paths. At jump points the right/left version difference is at most $1/\sqrt{k}$; on $[r_n,1]$ its weighted norm is $O_{\Pp}(k^{d-1/2}m_n^{-d})=o_{\Pp}(1)$. Alternatively the weighted bound for either version follows by taking one-sided limits of the coupled paths. Finally $b_n\to_{\Pp}1$; the dilation $s\mapsto b_ns$ preserves the weighted bound and $\sup_{0<s\le1}s^{-d}|W(b_ns)-W(s)|\to_{\Pp}0$ by compact continuity and the Brownian version of \eqref{eq:dyadic}. No independence of $b_n$ and the Wiener paths is needed.
\end{proof}

\begin{proposition}[Localized MNS decomposition]\label{prop:mns}
For $1/4<\eta<1/2$ and every $\epsilon>0$, under first-order conditions,
\begin{align}
 -Q_F(x)\sqrt{k}\Delta_n(x)&=J_n(x)+R_n(x),\\
 \sup_{1\le x\le T_n}x^{-(2\eta-p)/\gamma-\epsilon}|R_n(x)|&=o_{\Pp}(1),\\
 \sup_{1\le x\le T_n}|\Delta_n(x)|&=O_{\Pp}(m_n^{-1/2}).\label{eq:hazmax}
\end{align}
Moreover, for every $1/2<b<1$,
\begin{equation}\label{eq:hazweighted}
 \sup_{1\le x\le T_n}s_x^b\sqrt{k}|\Delta_n(x)|=O_{\Pp}(1).
\end{equation}
The Gaussian process $J_n$ is the one defined in the main paper, following \citet{MNS2025}. All hazard bounds hold also up to $U_H(2n/m_n)/t_n$.
\end{proposition}
\begin{proof}
Write $Y_n(v)=\bH_n(v-)$, $e_n(v)=\bH(v)-Y_n(v)$. The exact inverse-denominator expansion\[Y_n^{-1}=\bH^{-1}+e_n\bH^{-2}+\frac{e_n^2}{\bH^2Y_n}\] gives $\Delta_n=L_1+L_2+L_3+L_4$, where every integral below is over $(t_n,xt_n]$:
\begin{align*}
 L_1&=\int\bH^{-1}d(H_n^{(1)}-H^{(1)}),&
 L_2&=\int e_n\bH^{-2}dH^{(1)},\\
 L_3&=\int e_n\bH^{-2}d(H_n^{(1)}-H^{(1)}),&
 L_4&=\int e_n^2/(\bH^2Y_n)\,dH_n^{(1)}.
\end{align*}
All inverse risk factors are taken as zero when $Y_n=0$; we work on events where the risk is positive on the whole interval.

\emph{Linear terms.} Put $s=s_x$, $z_n(v)=z_h(b_nv)$ and $c_n=b_n^{-1}$. Stieltjes integration by parts gives
\[
 \sqrt{k}L_1=c_n\left\{\alpha_{n,1}(t_n)-\frac{\alpha_{n,1}(xt_n)}s+\int_s^1\alpha_{n,1}(z_n(v))v^{-2}dv\right\}.
\]
With $b=2\eta$ and $d=1-b$, Lemma~\ref{lem:coupling} replaces each empirical term by its Wiener counterpart at a total cost $o_{\Pp}(s^{-b})$, uniformly. In fact $\int_s^1 v^{d-2}dv\le s^{d-1}/(1-d)$, and $c_n-1=o_{\Pp}(1)$.
Define $D_n(v)=\overline H^{(1)}(z_n(v))/\bH(t_n)$. The reversed Stieltjes substitution gives
\[
 \sqrt{k}L_2=-c_n\int_s^1v^{-2}\{\alpha_{n,1}(z_n(v))+\alpha_{n,0}(z_n(v))\}\,dD_n(v).
\]
The left-risk convention makes no difference in this population integral because $H^{(1)}$ is continuous. The measures satisfy\[0\le dD_n\le dv,\qquad D_n(v)\to pv\] uniformly by Lemma~\ref{lem:tail}. Thus the coupling error is again $o_{\Pp}(s^{-b})$. To replace $dD_n$ by $p\,dv$, on $[u,1]$ approximate the continuous Brownian integrands by a deterministic mesh of step functions; the error is bounded by their modulus of continuity times the total mass, and the step-function integral difference tends to zero uniformly over lower endpoints by uniform convergence of $D_n$. On $(0,u]$, take $1-b<d'<1/2$; the Brownian paths satisfy $|\mathbf W_{n,j}(v)|\le O_{\Pp}(1)v^{d'}$. Hence the weighted error there is bounded by $O_{\Pp}(u^{b+d'-1})$ since $dD_n\le dv$. Taking $u\downarrow0$ yields
\[
\begin{gathered}
\sqrt{k}(L_1+L_2)=\mathcal A_n(s)+o_{\Pp}(s^{-b}) ,\\
\mathcal A_n(s)=\mathbf W_{n,1}(1)-\frac{\mathbf W_{n,1}(s)}s+\int_s^1\frac{q\mathbf W_{n,1}(v)-p\mathbf W_{n,2}(v)}{v^2}dv.
\end{gathered}
\]

\emph{Quadratic terms.} Let $\ell=U_H(n/(2k))$, $u^+=U_H(2n/m_n)$. With probability tending to one, $\ell\le t_n\le u_n\le u^+$. The relative risk error on $[0,u^+]$ is $O_{\Pp}(m_n^{-1/2})$. Indeed the same symmetrization argument as in Lemma~\ref{lem:coupling} gives\[\Ep\sup_{\bH(z)\le v}|\bH_n(z)-\bH(z)|^2\le16v/n,\qquad v\le1.\] On dyadic probability blocks $[2^j v_0,2^{j+1}v_0]$ with $v_0=m_n/(2n)$, Markov's inequality for the relative error at level $C/\sqrt{nv_0}$ gives a bound proportional to $C^{-2}2^{-j}$; summing proves the assertion. Left limits add at most $1/(nv_0)$. Write $\Omega_n$ for the event $\bH/2\le Y_n\le2\bH$ on this interval. Then $\Pp(\Omega_n)\to1$.

With $N_n=H_n^{(1)}$, the independent-censoring counting-process representation is\[dN_n=Y_n d\Lambda_F+dM_n,\qquad d\langle M_n\rangle=n^{-1}Y_n d\Lambda_F.\] This follows by summing the individual failure-count martingales: conditional on being at risk the failure hazard is $d\Lambda_F$; different individuals have no simultaneous jumps under continuity. Since $dH^{(1)}=\bH d\Lambda_F$,
\[
 L_3=\int e_n\bH^{-2}dM_n-\int e_n^2\bH^{-2}d\Lambda_F.
\]
To treat random endpoints, first integrate from $\ell$ to deterministic $z\le u^+$ and bound the running maximum; increments between two random endpoints are bounded by twice this maximum. Localize the martingale integrand to $Y_n\le2\bH$. At each deterministic $v$, $\Ep e_n(v)^2\le \bH(v)/n$. Doob's inequality gives
\[
 \Ep\sup_{\ell\le y\le z}\left|\int_\ell^y\frac{e_n\ind_{\{Y_n\le2\bH\}}}{\bH^2}dM_n\right|^2\le \frac{8}{n^2}\int_\ell^z\frac{d\Lambda_F}{\bH^2}\le\frac4{n^2\bH(z)^2}.
\]
Here $d\Lambda_F\le d\Lambda_H=-d\log\bH$. The nonnegative drift has expectation at most $1/(n\bH(z))$ by the same inequality. The compensator identity for $L_4$ gives
\[
 \Ep\int_\ell^z\frac{e_n^2}{\bH^2Y_n}dN_n=\Ep\int_\ell^z\frac{e_n^2}{\bH^2}\ind_{\{Y_n>0\}}d\Lambda_F\le\frac1{n\bH(z)}.
\]
Use dyadic deterministic levels $v_j=2^{-j}$ in the probability coordinate $n\bH(z)/k$, from 2 to $m_n/(2k)$, including a final partial block. Since $b_n\to_{\Pp}1$, these levels are comparable to $s_x$. For any $\epsilon>0$, the union bound for the scaled supremum $\sqrt{k}s_x^b(|L_3|+|L_4|)$ gives, apart from the single probability of $\Omega_n^c$ and the endpoint event, a sum bounded by
\[
 C_\epsilon\sum_j\left(k^{-1/2}v_j^{b-1}+k^{-1}v_j^{2b-2}\right)
 \le C_\epsilon\left(k^{1/2-b}m_n^{b-1}+k^{1-2b}m_n^{2b-2}\right)\to0.
\]
Thus, uniformly,\[\sqrt{k}\Delta_n=\mathcal A_n(s_x)+o_{\Pp}(s_x^{-b}).\] Brownian bounds with any $d'<1/2$ give $\mathcal A_n(s)=O_{\Pp}(s^{d'-1})$ uniformly, proving \eqref{eq:hazweighted}.

\emph{Return to the survival coordinate.} Compact convergence of the ratios $Q_F,s_x$ to $x^{-p/\gamma},x^{-1/\gamma}$ and Brownian equicontinuity yield the approximation on compact $x$-intervals. On the complement, choose $1-b<d'<1/2$ and sufficiently small Potter margins. Both $Q_F\mathcal A_n(s_x)$ and $x^{-p/\gamma}\mathcal A_n(x^{-1/\gamma})$, multiplied by $x^{-(b-p)/\gamma-\epsilon}$, are bounded by $O_{\Pp}(x^{(1-d'-b)/\gamma})$ after decreasing the margins if needed. This bound vanishes as the compact cutoff increases. The remainder also has vanishing weighted norm by Potter. The identity $J_n(x)=-x^{-p/\gamma}\mathcal A_n(x^{-1/\gamma})$ follows by the substitution $v=u^{-1/\gamma}$ and is exactly the sum $J_{1n}+J_{2n}$ of the main paper. In particular, for any $\epsilon>0$,\[|J_n(x)|\le O_{\Pp}(1)x^{(1/2-p)/\gamma+\epsilon}.\]

\emph{Uniform unscaled hazard bound.} On $\Omega_n$, $\Delta_n(x)=\int_{(t_n,xt_n]}dM_n/Y_n$. The martingale from $\ell$ with predictable integrand $\ind_{\{Y_n\ge\bH/2\}}/Y_n$ has quadratic variation up to $u^+$ at most $2/(n\bH(u^+))=4/m_n$. Doob's inequality yields \eqref{eq:hazmax} on the enlarged domain as well. All arguments above apply to this enlarged interval, whose smallest expected risk count is $m_n/2$.
\end{proof}

\begin{lemma}[Weighted power linearization]\label{lem:linear}
For $0<\zeta<\beta-1/2$,
\[
 \sup_{1\le x\le T_n}x^{\zeta/\gamma}\sqrt{k}|Q_n(x)^a-Q_F(x)^a+aQ_F(x)^a\Delta_n(x)|=o_{\Pp}(1).
\]
\end{lemma}
\begin{proof}
By \eqref{eq:hazmax}, $\sup|\Delta_n|=o_{\Pp}(1)$, so the Taylor remainder is bounded by $C Q_F^a\Delta_n^2$ on events of probability tending to one. Potter bounds imply $x^{\zeta/\gamma}Q_F^a\le Cs_x^{\beta-\zeta-\epsilon}$ for arbitrarily small $\epsilon>0$. Equation \eqref{eq:hazweighted} bounds the weighted scaled remainder by\[ O_{\Pp}(k^{-1/2})\sup_{r_n\le s\le1}s^{\beta-\zeta-\epsilon-2b}.\] Choose $b>1/2$ close enough to $1/2$ and $\epsilon$ small enough that $D=\beta-\zeta-\epsilon-2b>-1/2$. If $D\ge0$ the supremum is one; otherwise it is $r_n^D$ and $k^{-1/2}r_n^D=O_{\Pp}(k^{-1/2-D}m_n^D)=o_{\Pp}(1)$. This proves the result.
\end{proof}

\begin{lemma}[Second-order and terminal replacement]\label{lem:second}
Under the second-order assumption on $F$ and $\sqrt{k}A_1(h)=O(1)$, for $0<\zeta<\beta-1/2$,
\[
 \sup_{1\le x\le T_n}x^{\zeta/\gamma}\left|\sqrt{k}(Q_F(x)^a-x^{-\beta/\gamma})-\sqrt{k}A_1(h)\mathbb B_\beta(x)\right|=o_{\Pp}(1).
\]
\end{lemma}
\begin{proof}
The second-order uniform convergence theorem with Potter margins (\citealp[Chapter 2]{deHF06}) yields
\[
 Q_F(x)=x^{-1/\gamma_1}[1+A_1(h)\{d_{\tau_1}(x)+e_n^*(x)\}],\quad
 \sup_{x\ge1}x^{-\epsilon}|e_n^*(x)|=o_{\Pp}(1).
\]
The random threshold replacement uses $t_n/h\to_{\Pp}1$ and $A_1(t_n)/A_1(h)\to_{\Pp}1$. To see why the uniform statement follows, the second-order Potter bound controls the normalized relative increment by $Cx^{\epsilon/2}$ for $x\ge1$, while compact convergence supplies the limit; multiply by $x^{-\epsilon}$ and let the compact cutoff grow. For $v\ge-1$ and fixed $a>1$,
\[
 |(1+v)^a-1-av|\le C_a v^2(1+|v|)^{(a-2)_+}.
\]
Apply this with $v=A_1(h)(d_{\tau_1}+e_n^*)$, for which $|v|\le O_{\Pp}(|A_1(h)|)x^\epsilon$. Choose $\epsilon$ sufficiently small that $\zeta-\beta+\gamma\epsilon\max(2,a)<0$. Then the weighted quadratic remainder is\[O_{\Pp}(\sqrt{k}A_1(h)^2)=o_{\Pp}(1),\] and the linear remainder is negligible by the weighted bound for $e_n^*$. Finally $|\mathbb B_\beta(x)|\le Cx^{-\beta/\gamma}(1+\log x)$, implying the weighted terminal assertion in Theorem 1.
\end{proof}

\begin{proposition}[Exact representation, discretization, and endpoint]\label{prop:discrete}
The powered product-limit companion equals $\Phi_n(\alpha_n,\widehat T_n)$ exactly, and both discrete jump conventions differ from this integral by $O_{\Pp}(k^{-1})$. For any sufficiently small $\epsilon>0$,
\[
 |\Phi_n(\alpha_n,\widehat T_n)-\Phi_n(\alpha_n,T_n)|=O_{\Pp}\big((m_n/k)^{\beta-\epsilon}\log(k/m_n)\big).
\]
\end{proposition}
\begin{proof}
Use\[h_{i,n}=\frac{\delta_{[n-i+1:n]}}i,\qquad R_{i,k}=\exp\left(-\sum_{j=i+1}^k h_{j,n}\right),\qquad \ell_{i,k}=\log(Z_{n-i+1:n}/t_n).\] The descending jump in $Q_n^{\alpha_n}$ at rank $i$ is $R_{i,k}^{\alpha_n}(1-e^{-\alpha_n h_{i,n}})$. Integration by parts gives the exact identity
\[
 \Phi_n(\alpha_n,\widehat T_n)=\alpha_n\sum_{i=m_n}^kR_{i,k}^{\alpha_n}(1-e^{-\alpha_n h_{i,n}})\ell_{i,k}.
\]
Since, for $y\ge0$,\[0\le y-(1-e^{-y})\le y^2/2,\qquad 0\le(1-e^{-y})-ye^{-y}\le y^2/2,\] the absolute difference for either convention is bounded by $(\alpha_n^3/2)\sum i^{-2}R_{i,k}^{\alpha_n}\ell_{i,k}$. The enlarged domain in Proposition~\ref{prop:mns} covers $\widehat T_n$ with probability tending to one; its hazard bound and Lemma~\ref{lem:tail} yield, uniformly for $m_n\le i\le k$,
\[
 R_{i,k}^{\alpha_n}\le C(i/k)^{\beta-\epsilon},\qquad \ell_{i,k}\le C(1+\log(k/i)).
\]
The first bound follows by applying Potter to $Q_F$ in the survival coordinate and using $p\alpha_n\to_{\Pp}\beta$; replacing right by left limits changes the power by at most $\exp(\alpha_n/m_n)$. The second bound follows from the upper Potter bound on $U_H$ and the uniform relative order-statistic bound. With $b=\beta-\epsilon>1$, the discrete sum is bounded by\[Ck^{-b}\sum_{i=1}^k i^{b-2}(1+\log(k/i))=O(k^{-1}),\] by the integrability of $u^{b-2}(1-\log u)$ at zero and elementary integral comparison.

For $T^- = T_n\wedge\widehat T_n$ and $T^+=T_n\vee\widehat T_n$, monotonicity gives
\[
 0\le\Phi_n(\alpha_n,T^+)-\Phi_n(\alpha_n,T^-)\le\alpha_n Q_n(T^-)^{\alpha_n}\log T^+.
\]
Lemma~\ref{lem:tail}, Potter bounds and \eqref{eq:hazmax} give $Q_n(T^-)^{\alpha_n}=O_{\Pp}((m_n/k)^{\beta-\epsilon})$ and \allowbreak $\log T^+=O_{\Pp}(\log(k/m_n))$. This proves the claim, with no rate assumption on $\widehat u_n/u_n-1$. The boundary condition of Theorem 3 makes this bound $o_{\Pp}(k^{-1/2})$; for consistency it is simply $o_{\Pp}(1)$.
\end{proof}

\begin{lemma}[Vanishing derivative with respect to the exponent]\label{lem:derivative}
For a compact neighborhood $K\subset(1,\infty)$ of $a$ and either $T=T_n$ or $T=\widehat T_n$,
\[
 \sup_{\alpha\in K}|\partial_\alpha\Phi_n(\alpha,T)|=o_{\Pp}(1).
\]
If $\sqrt{k}(\widehat p_k-p)=O_{\Pp}(1)$, then\[\sqrt{k}(\Phi_n(\alpha_n,T)-\Phi_n(a,T))=o_{\Pp}(1).\]
\end{lemma}
\begin{proof}
Differentiate under the finite integral:
\begin{align*}
 \partial_\alpha\Phi_n(\alpha,T)&=\int_1^T Q_n(x)^\alpha(1+\alpha\log Q_n(x))\frac{dx}{x}\\
 &\quad-Q_n(T)^\alpha(1+\alpha\log Q_n(T))\log T.
\end{align*}
The enlarged hazard bound gives $\sup|\log Q_n-\log Q_F|=o_{\Pp}(1)$. For $a_-=\min K$, Potter bounds provide the integrable envelope $Cx^{-a_-/\gamma_1+\epsilon}(1+\log x)/x$, uniformly over $K$, and the terminal term is bounded by $CT^{-a_-/\gamma_1+\epsilon}(1+\log T)^2=o_{\Pp}(1)$. Compact uniform convergence and tail domination give uniform convergence to the derivative of $\alpha\int_1^\infty x^{-\alpha/\gamma_1}dx/x=\gamma_1$, which is zero. The mean-value theorem and $\sqrt{k}(\alpha_n-a)=O_{\Pp}(1)$ prove the final assertion. For consistency, the same argument uses only $\alpha_n-a=o_{\Pp}(1)$.
\end{proof}

\begin{lemma}[Integrated Gaussian term with a fixed cutoff]\label{lem:gauss}
For $N_n(\beta)$ defined in the main paper,
\[
\begin{gathered}
a\int_1^{T_n}\mathbb G_n^{(\beta)}(x)\frac{dx}{x}=\gamma_1N_n(\beta)+o_{\Pp}(1) ,\\
\operatorname{Var}N_n(\beta)=\frac{\beta^2}{p(2\beta-1)}.
\end{gathered}
\]
\end{lemma}
\begin{proof}
In the coordinate $s=x^{-1/\gamma}$, $\mathbb G_n^{(\beta)}(x)=-as^\beta\mathcal A_n(s)$. Since $\|\mathcal A_n(s)\|_2\le Cs^{-1/2}$ by Minkowski's inequality and the Wiener covariance, the Gaussian integral over $[1,\infty)$ exists in $L^2$. For deterministic $T$, the squared $L^2$ tail is at most $CT^{-(2\beta-1)/\gamma}$. Also\[\Ep\int_T^\infty|\mathbb G_n^{(\beta)}(x)|\frac{dx}{x}\le CT^{-(\beta-1/2)/\gamma}.\] Therefore the integral up to the random endpoint differs from the full Gaussian integral by $o_{\Pp}(1)$: split at fixed $T$, bound the tail by the absolute integral on $\{T_n>T\}$ and use $\Pp(T_n\le T)\to0$, then let $T\to\infty$. No independence of the endpoint and Wiener processes is assumed.

Put $I_j=\int_0^1s^{\beta-2}\mathbf W_{n,j}(s)ds$. Fubini's theorem (justified by the same absolute moment bounds) gives
\begin{align*}
 -a^2\gamma\int_0^1s^{\beta-1}\mathcal A_n(s)ds
 &=a^2\gamma\left\{I_1-\frac{\mathbf W_{n,1}(1)}\beta+\frac{p I_2-q I_1}\beta\right\}\\
 &=\gamma_1\left\{\frac{\beta(p+\beta-1)}p I_1+\beta I_2-\frac\beta p\mathbf W_{n,1}(1)\right\}.
\end{align*}
The Gaussian vector has
\begin{align*}
\operatorname{Var} I_1&=\frac{2p}{\beta(2\beta-1)},\\
\operatorname{Var} I_2&=\frac{2q}{\beta(2\beta-1)},\\
\operatorname{Cov}(I_1,\mathbf W_{n,1}(1))&=\frac{p}{\beta},
\end{align*}
and zero cross-covariances between the two Wiener processes. Substitution gives
\[
 \begin{aligned}\operatorname{Var}N_n(\beta)&=\frac{2\beta(p+\beta-1)^2}{p(2\beta-1)}+\frac{2\beta q}{2\beta-1}\\ &\quad+\frac{\beta^2}p-\frac{2\beta(p+\beta-1)}p\\ &=\frac{\beta^2}{p(2\beta-1)}.\end{aligned}
\]
For an independent check, set\[W=\mathbf W_{n,1}+\mathbf W_{n,2},\qquad B=q\mathbf W_{n,1}-p\mathbf W_{n,2}.\] These independent Gaussian processes have covariance kernels $s\wedge t$ and $pq(s\wedge t)$. Integration by parts gives
\[
 \mathcal A_n(s)=p\{W(1)-W(s)/s\}+\int_s^1 v^{-1}dB(v).
\]
The covariance of the first bracket is $(s\vee t)^{-1}-1$, as follows by expanding the four covariances; the stochastic integral has covariance $pq\{(s\vee t)^{-1}-1\}$. Hence\[\operatorname{Cov}(\mathcal A_n(s),\mathcal A_n(t))=p\{(s\vee t)^{-1}-1\}.\] In particular $\mathcal A_n$ has the law of $\sqrt p\,W_0(s^{-1}-1)$ for a standard Wiener process $W_0$. Since $\mathbb G_n^{(\beta)}=-as^\beta\mathcal A_n(s)$ this yields the representation in the main paper. The variance of the integrated term is also
\[
 a^4\gamma^2 p\frac{2}{\beta}\int_0^1 t^{2\beta-1}(t^{-1}-1)dt
 =\frac{\gamma_1^2\beta^2}{p(2\beta-1)}.
\]

\end{proof}

\clearpage
\section{Simulation figures}\label{app:simfig}
\begin{figure}[!ht]
\centering
\includegraphics[width=\textwidth]{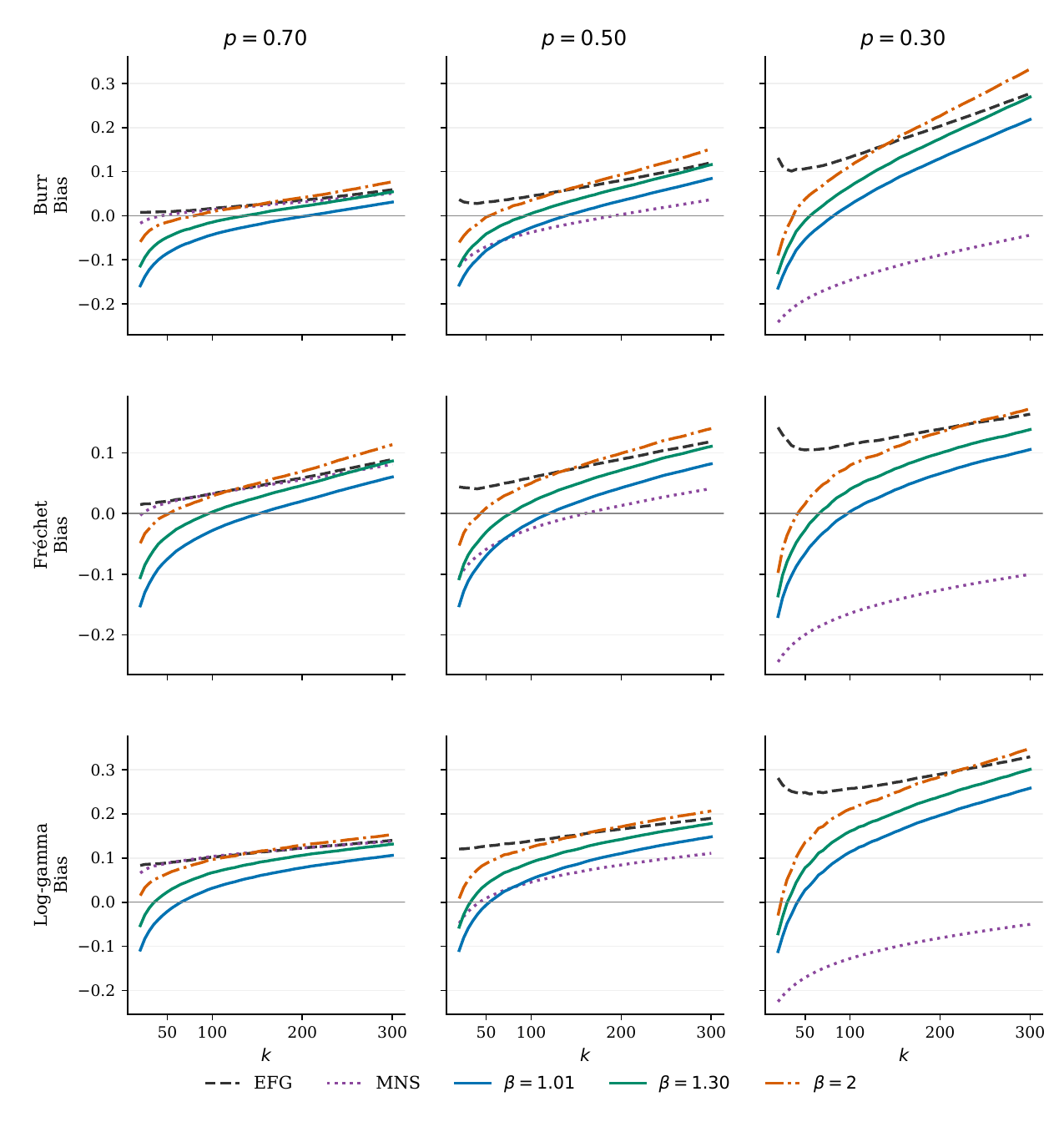}
\caption{Signed empirical bias versus $k$ for $\gamma_1=0.4$, $n=1000$ and $B=2000$. Rows: Burr, Fr\'echet and log-gamma; columns: $p=0.70,0.50,0.30$.}
\label{fig:bias04}
\end{figure}
\clearpage
\begin{figure}[p]
\centering
\includegraphics[width=\textwidth]{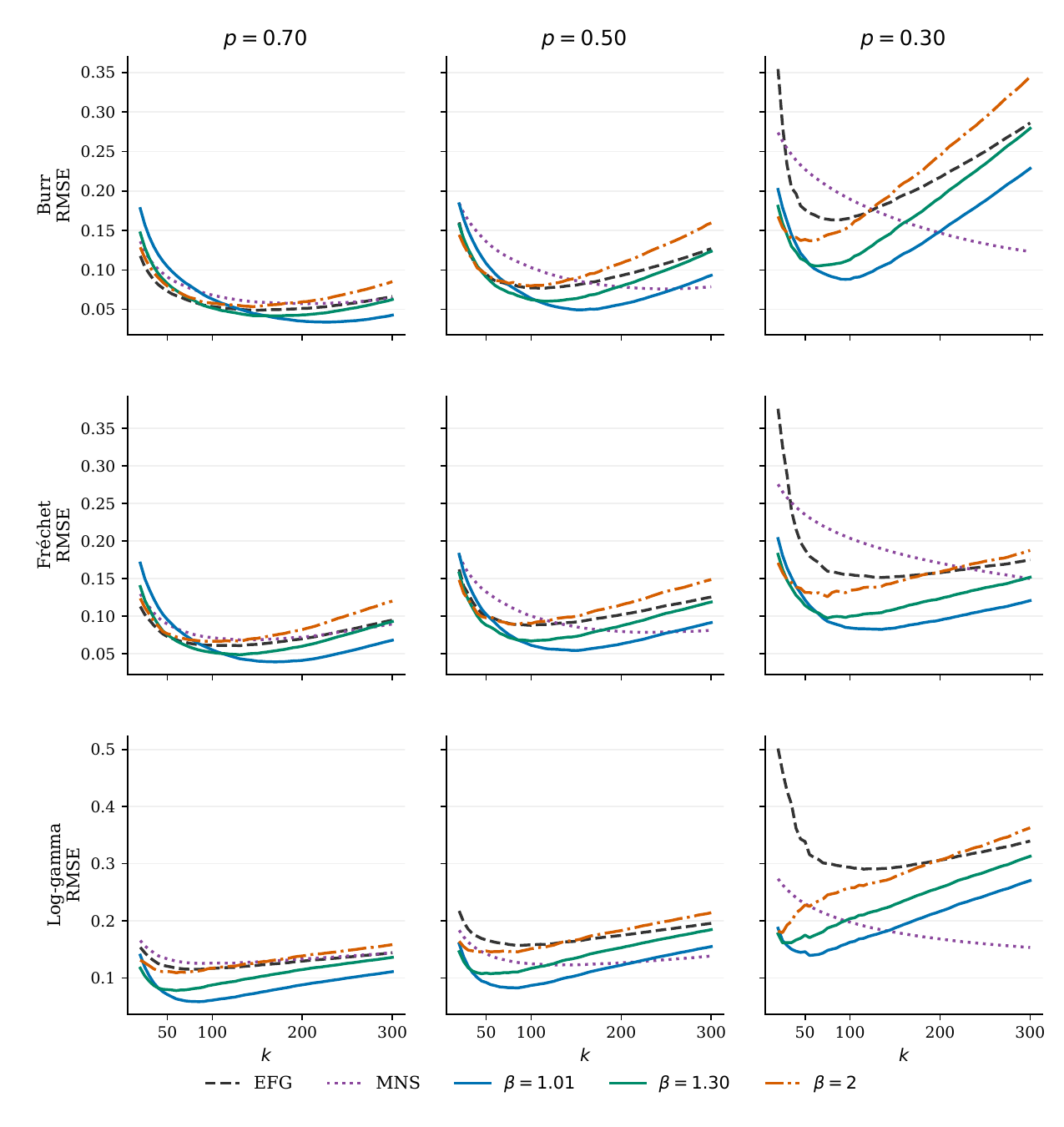}
\caption{Empirical RMSE versus $k$ for $\gamma_1=0.4$, $n=1000$ and $B=2000$. Rows: Burr, Fr\'echet and log-gamma; columns: $p=0.70,0.50,0.30$.}
\label{fig:rmse04}
\end{figure}
\begin{figure}[p]
\centering
\includegraphics[width=\textwidth]{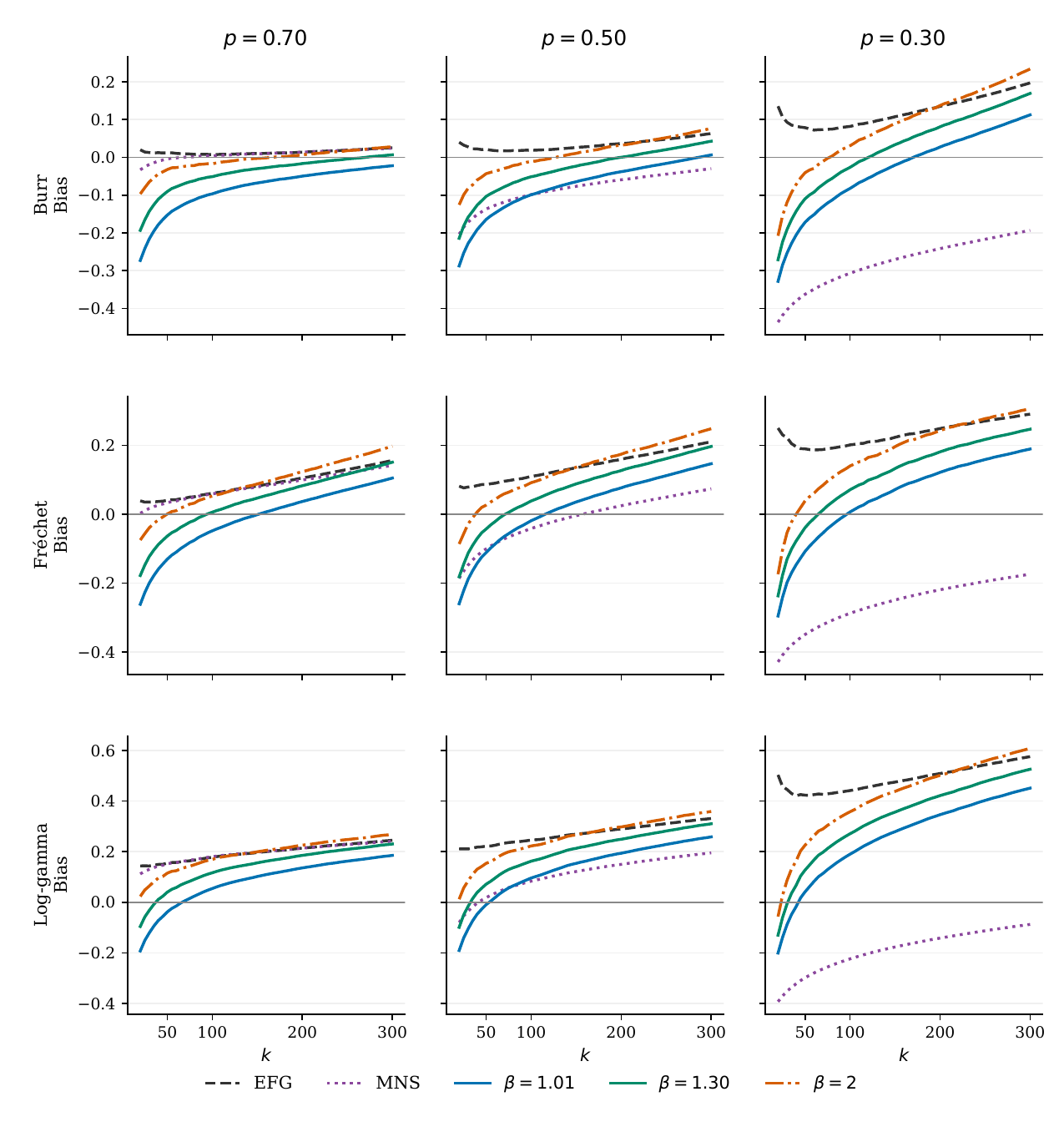}
\caption{Signed empirical bias versus $k$ for $\gamma_1=0.7$, $n=1000$ and $B=2000$. Rows: Burr, Fr\'echet and log-gamma; columns: $p=0.70,0.50,0.30$.}
\label{fig:bias07}
\end{figure}
\begin{figure}[p]
\centering
\includegraphics[width=\textwidth]{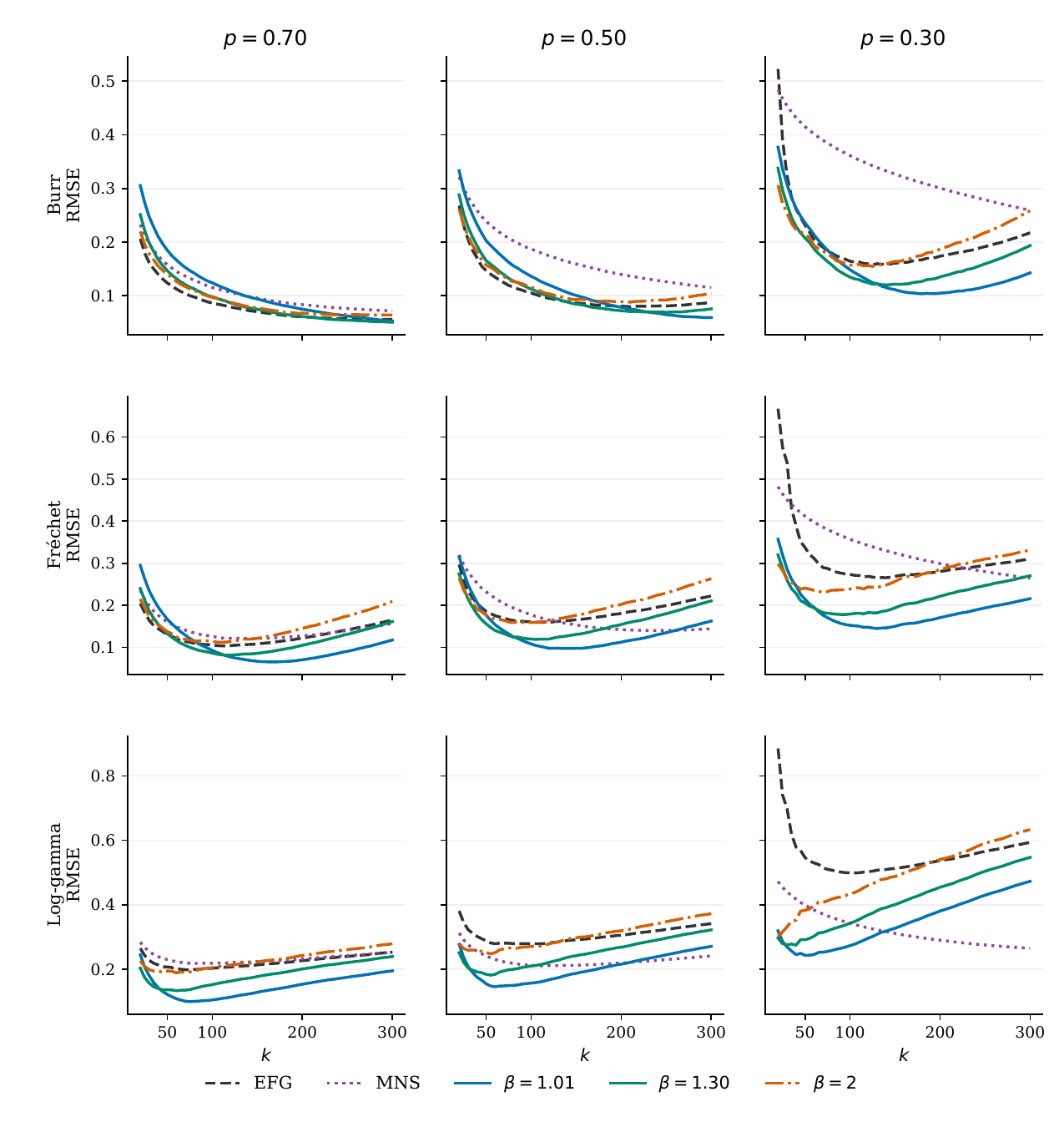}
\caption{Empirical RMSE versus $k$ for $\gamma_1=0.7$, $n=1000$ and $B=2000$. Rows: Burr, Fr\'echet and log-gamma; columns: $p=0.70,0.50,0.30$.}
\label{fig:rmse07}
\end{figure}
\clearpage
\clearpage
	\renewcommand{\refname}{References}
	\small
	\setstretch{1.0}
\setlength{\parskip}{0pt}
	\setlength{\bibsep}{0pt}
	\setlength{\bibhang}{0.5em}
	
\normalsize

\end{document}